\documentclass[11pt,twoside]{amsart}

\usepackage[latin1]  {inputenc}%
\usepackage[T1]      {fontenc }%
\usepackage          {amsmath }%
\usepackage          {amsfonts}%
\usepackage          {amssymb }%
\usepackage          {amsthm  }%
\usepackage          {a4wide  }%
\usepackage          {url     }%
\usepackage          {tikz    }%
\usepackage[bookmarks=false,pdfborder={0 0 0.05}]{hyperref}
\usepackage[all]{xy}
\usepackage{lmodern}
\usepackage{tikz-cd}

\usepackage{enumerate, amsmath, amsfonts, amssymb, amsthm,  wasysym, graphics, graphicx, xcolor, frcursive,comment,bbm}

\usepackage{etex}

\definecolor{darkblue}{rgb}{0.0,0,0.7} % darkblue color
\definecolor{darkred}{rgb}{0.7,0,0} % darkred color
\newtheorem{innercustomthm}{Theorem}
\newenvironment{customthm}[1]
  {\renewcommand\theinnercustomthm{#1}\innercustomthm}
  {\endinnercustomthm}

\usepackage{hyperref}
\usepackage[all]{xy}
\usepackage[T1]{fontenc}

\usepackage{MnSymbol}

\def\defn#1{{\sf #1}}

\DeclareMathOperator{\XX}{\mathbb{X}}
\DeclareMathOperator{\COH}{coh}
\newcommand{\RR}{\mathbb R}
\newcommand{\ZZ}{\mathbb Z}

\newcommand{\spanz}{\mathrm{span}_{\mathbb{Z}}}
\newcommand{\spanr}{\mathrm{span}_{\mathbb{R}}}

\DeclareMathOperator{\Red}{Red}
\DeclareMathOperator{\Span}{span}

\DeclareMathOperator{\idop}{id}

\DeclareMathOperator{\Redd}{Red^{gen}}

\DeclareMathOperator{\GL}{GL}
\DeclareMathOperator{\TR}{tv}

\DeclareMathOperator{\END}{End}
\DeclareMathOperator{\Sym}{Sym}

\DeclareMathOperator{\NN}{\mathbb{N}}

\DeclareMathOperator{\MOD}{mod}

\DeclareMathOperator{\EXT}{Ext}

\DeclareMathOperator{\Thi}{Thick}

\DeclareMathOperator{\Mov}{Mov}
\DeclareMathOperator{\fin}{fin}

\newtheorem{theorem}{Theorem}[section]
\newtheorem{corollary}[theorem]{Corollary}
\newtheorem{proposition}[theorem]{Proposition}
\newtheorem{Lemma}[theorem]{Lemma}

\theoremstyle{definition}
\newtheorem{definition}[theorem]{Definition}
\newtheorem{propdef}[theorem]{Proposition and Definition}
\newtheorem{remark}[theorem]{Remark}

\title[Extended Weyl groups, Hurwitz transitivity and weighted projective lines II]{Extended Weyl groups, Hurwitz transitivity and weighted projective lines II: a uniform approach}

\author[B.~Baumeister]{Barbara Baumeister}
\address{Barbara Baumeister, Universit\"at Bielefeld, Germany}
\email{b.baumeister@math.uni-bielefeld.de}
\thanks{}

\author[P.~Wegener]{Patrick Wegener}
\address{Patrick Wegener, Leibniz Universit\"at Hannover, Germany}
\email{patrick.wegener@math.uni-hannover.de}
\thanks{}

\author[S.~Yahiatene]{Sophiane Yahiatene}
\address{Sophiane Yahiatene, Universit\"at Bielefeld, Germany}
\email{syahiate@math.uni-bielefeld.de}

\date{\today}

\begin{document}

\begin{abstract}
In this paper, which is a continuation of \cite{BWY21}, we establish a uniform approach to extended Weyl groups. First, we focus on extended Weyl groups of domestic or wild type. We start with an extended Coxeter-Dynkin diagram, attach to it a set of roots (vectors) in an $\RR$-space and define the extended Weyl groups as groups that are generated by the of set reflections $S$ related to these roots, the so called simple reflections.
We relate to such a system $(W,S)$ a generalized root system, and determine the structure of $W$. 
%In particular we present a  normal form for the elements of $W$. 
Then we show that every extended Weyl system of domestic or wild type is an extended Coxeter system of star type (\cite{Looi80, Lek}).
We further consider the hyperbolic cover of an extended Weyl system, and show
that the hyperbolic covers of the extended Weyl systems are precisely the extended Coxeter systems  of star type.

Further we define  Coxeter transformations  in $W$. We show 
the transitivity of the Hurwitz action  on the set of reduced reflection factorizations of a Coxeter transformation (the dual Matsumoto property) in an extended Coxeter systems of star type, where
the reflections are the conjugates of the simple reflections in $W$ (Theorem~\ref{thm:AllgAussage}) 
This then also proves the respective statement for the extended Weyl systems of domestic or wild type (Theorem~\ref{thm:main}), which is failing for extended Weyl systems of tubular type (see \cite{BWY21} or \cite{BMW25}).
Therefore, it turns out, that the hyperbolic cover of an extended Weyl system is the appropriate framework of our considerations (see Theorem~\ref{thm:AllgAussagehyperCover}).

We give two applications of Theorem~\ref{thm:AllgAussagehyperCover} and of the results in \cite{BWY21, BMW25}. In the context of representation theory of algebras, %\cite[Theorems~1.1 and 1.4]{BWY21} and 
Theorem~\ref{thm:AllgAussagehyperCover} is the last step that establishes an isomorphism between the poset of  thick subcategories that are generated by  exceptional sequences of a hereditary connected ext-finite abelian $k$-category with a tilting
object for an algebraically closed field $k$ of characteristic $0$ and the poset of elements in the  hyperbolic cover of an extended Weyl group  that are below a Coxeter transformation with respect to the absolute order (Theorem \ref{main_generalCor}). 
The second application concerns the theory of unimodal singularities (Theorem~\ref{thm:dist_bases}). In particular, we provide an answer to a question of Brieskorn \cite[Question 4]{Ebe19} for the classical monodromy operator in the case of hyperbolic singularities.

\end{abstract}

\maketitle

\tableofcontents

\section{Introduction }

Affine Coxeter systems are well studied, for instance group theoretical, topological or combinatorial properties of them are known (see for instance \cite{DL11, EE98, Hum90, McC15, PW17a}). In \cite{BWY21, BMW25},  and in this paper we propose the systematic study of a generalization of simply laced  affine Coxeter systems, more precisely the study of \defn{extended Weyl groups} and \defn{extended Weyl systems}.
An extended Weyl group is a group that is  generated by a certain set of reflections of a finite dimensional  $\RR$-space, which is equipped with a bilinear form (see Section~\ref{sec:Notation}). 
There are three different types of extended Weyl groups, those of \defn{domestic}, \defn{tubular} and of \defn{wild} type corresponding to the three possible types of the related bilinear form. The extended Weyl groups of domestic type are precisely the simply laced affine Coxeter groups (see Remark \ref{Rem:RootSystem}(c)), and those of tubular type are \defn{elliptic Weyl groups} (see \cite{BWY21, KL86, Sai85}).

Apart from the domestic case, extended Weyl groups are not Coxeter groups. Nevertheless, they share several properties with Coxeter groups. For instance, every extended Weyl group $W$ has a linear faithful representation, and is related to a generalized root system, which contains a ``\defn{simple system}'' of roots, and the set $S$ of reflections with respect to the roots in the simple system generates $W$. 

Therefore, each extended Weyl group is  defined uniquely by a diagram, the so-called \defn{extended Coxeter Dynkin diagram}, and contains special elements, the \defn{Coxeter transformations}, which are the analogues of Coxeter elements in a Coxeter group (see Section~\ref{sec:Notation}). 

Extended Weyl systems have already been partially studied in \cite{KL86, Len99, STW16} as well as in \cite{PW17b, BWY21}. The origin of Kluitmann's (and Brieskorn's) \cite{KL86} motivation to study these groups was his interest in simple elliptic singularities of types $E^{(1,1)}_{6},~E^{(1,1)}_{7}$ and $E^{(1,1)}_{8}$, as their monodromy groups are extended Weyl groups. In order to determine the distinguished bases of the corresponding Milnor lattices he considered an action of the braid group, the \defn{Hurwitz action} (see Section~\ref{sec:Notation}), on the set of reduced reflection factorizations of a Coxeter transformation.

The motivation of the other authors was in the representation theory of algebras, more precisely the interest in  the categories of coherent sheaves over a weighted projective line $\COH(\XX)$ in the sense of Geigle and Lenzing \cite{GL87}, or respectively in the derived equivalent categories of finite dimensional modules over canonical algebras (see the definition of canonical algebras in \cite{RI84}). One can attach to such a  category a generalized root system as well as a reflection group, which is an extended Weyl group (see \cite[Section~2]{BWY21}). Also in the understanding of the category $\COH(\XX)$ the Hurwitz action of the braid group on the set of reduced factorizations of a Coxeter transformation into a product of  reflections plays an important role, as we will explain in this paper.

We study extended Weyl groups of tubular type, i.e. elliptic Weyl groups, in \cite{BWY21, BMW25}.  In this paper, which is a continuation of \cite{BWY21}, we focus on the extended Weyl groups of wild and  of domestic type. 
We  describe their group structure (Subsection~\ref{subsec:StructureW}) and introduce a normal form for its elements (Subsection~\ref{subsec:NormalForm}).
This enables us to establish a uniform approach to all the extended Weyl systems by considering them as 
extended Coxeter systems.
Looijenga and van der Lek were the first to study extended Coxeter systems (see \cite{Looi80, Lek, Lek2}).
 Given a crystallographic Coxeter system it is an extension of the Coxeter group by its root lattice (see Section~\ref{sec:ExtCoxGrpRootLattice}). We show that every extended Weyl system of domestic or wild type is an extended Coxeter system of star 
type, i.e. the Dynkin diagram of the involved Coxeter group is a star (Proposition~\ref{Prop:CharExtendedWeylNotTubular}).

We extend our results on the Hurwitz transitivity for Coxeter elements in Coxeter systems \cite{BDSW14} and Coxeter transformations in elliptic Weyl systems  to Coxeter transformations in extended Coxeter systems $(\mathcal{W}, \mathcal{S})$ of star type (Theorem~\ref{thm:AllgAussage}). 
The next theorem is then a corollary of Theorem~\ref{thm:AllgAussage}.
Let $T = \bigcup_{w \in \mathcal{W}} w \mathcal{S} w^{-1}$ be the set of reflections in $W$.

\begin{theorem}\label{thm:main}
Let $(W,S)$ be an extended Weyl system of domestic or wild type with simple system $S$ of size $n$, set of reflections $T$, and  Coxeter transformation $c$. Then the Hurwitz action is transitive on the set of reduced reflection factorizations of $c$, namely on the set $\Red_{T}(c)=\left\lbrace (t_{1},\ldots,t_{n})\in T^{n}\mid t_{1}\cdots t_{n}=c \right\rbrace$.
\end{theorem}

Let $\Redd(c)$ be the set of reduced reflection factorizations of a Coxeter transformation $c$ that generate $W$ (see Definition~\ref{Def:Generating}). Recall that 
in  an elliptic Weyl group, which is not of type $D_4^{(1,1)}$, the Hurwitz action is transitive on the set $\Redd(c)$ (see \cite{BWY21}). 
Furthermore, the authors proved that in this case $\Redd(c) \neq \text{Red}_{T}(c)$. In the non-elliptic case $\Redd(c) = \text{Red}_{T}(c)$ holds because of Theorem~\ref{thm:main}. 

In \cite{BMW25} the \defn{hyperbolic cover} of an elliptic Weyl system  is constructed, which was introduced by Saito in \cite{Sai74}. The hyperbolic cover is an extended Coxeter system of star type (see \cite{BMW25} or Theorem~\ref{Thm:Isomorphy}). 
In Section~\ref{} the hyperbolic cover of an extended Weyl system 
of domestic or of wild type is considered. We show that it
is isomorphic to the extended Weyl system itself (Proposition~\ref{Prop:hypcoverWild}). 
Hence we will also obtain  the next theorem as a corollary to 
Theorem~\ref{thm:AllgAussage}.

\begin{theorem}\label{thm:AllgAussagehyperCover}
Let $(\tilde{W}, \tilde{S})$ be the hyperbolic cover of an extemded Weyl system with set of reflections $\tilde{T}$, and a Coxeter transformation $\tilde{c}$. 
 Then the Hurwitz action is transitive on the set of reduced reflection factorizations $R_{\tilde{T}}(\tilde{c})$ of $\tilde{c}$.
\end{theorem}

In the last section we apply Theorem~\ref{thm:main} to the representation theory of hereditary categories. Every category of coherent sheaves over a weighted projective line is a hereditary category with a tilting object. The Grothendieck group of an ext-finite hereditary  abelian $k$-category for an algebraically closed field $k$ that has a tilting object is naturally equipped with a generalized root system and a reflection group. In particular, this is the case for the category of finitely generated $A$-modules for a finite dimensional  hereditary $k$-algebra $A$.

The reflection group attached to the category $\MOD(A)$, also called Weyl group, is a Coxeter group $W$ with set of simple reflections $S$,  which is obtained from a complete exceptional sequence. Such a sequence can be chosen such that it contains a complete set of representatives of the set of isomorphism classes of the finite dimensional simple $A$-modules (see \cite{Ri94, HK16}). As usual, the product of the elements of $S$ in arbitrary order is a Coxeter element of $(W,S)$.

According to Happel, the two just mentioned types of categories, that is, those of finitely generated modules of a finite dimensional hereditary $k-$algebra and those related to weighted projective lines,
are precisely up to derived equivalence the two types of hereditary connected ext-finite abelian $k-$categories with a tilting object (\cite[Theorem 3.1]{Hap01}).

Ingalls and Thomas \cite{IT09}, and Igusa, Schiffler and Thomas \cite{IS10} as well as Krause \cite{Kra12} and Hubery and Krause \cite{HK16} obtained a combinatorial description of the poset of thick subcategories generated by exceptional sequences of $\MOD(A)$ by providing an isomorphism (of posets) between this poset and a combinatorial object, the poset of noncrossing partitions in $W$. The latter poset consists of all the elements of $W$ that are in the interval $[1,c]$ with respect to some order relation on $W$, the so-called \defn{absolute order} (see \cite{BWY21}).

Besides our interest in the extended Weyl groups by themselves, our main goal of this and the previous paper \cite{BWY21} is to transfer the results for the module categories $\MOD(A)$ to the category of coherent sheaves over a weighted projective line, the other type of hereditary connected ext-finite abelian $k-$categories with a tilting object.

As a consequence of Theorem~\ref{thm:main}, \cite[Theorem 1.1]{BWY21} and \cite[Theorem~1.4]{BMW25}  we obtain the desired isomorphism between the poset of thick subcategories  of the category of coherent sheaves over a weighted projected line that are generated by an exceptional sequence, whose order relation is inclusion, and a subposet of the set of generalized non-crossing partitions $[\idop, c]$.
We denote by $\tilde{W}$ the hyperbolic cover of $W$ (see Proposition~\ref{Prop:hypcoverWild} and \cite{BMW25}).

\begin{theorem} \label{thm:mainRep}
Let $\XX$ be a weighted projective line over an algebraically closed field $k$ of characteristic zero, $\COH(\XX)$ the category of coherent sheaves over $\XX$, $\Phi$ the associated generalized root system, $\tilde{W}$ the hyperbolic  cover of the associated extended Weyl group $W$ and $c \in \tilde{W}$ a Coxeter transformation. Then there exists an isomorphism of posets
between
  \begin{itemize}
    \item the poset of thick subcategories of $\COH(\XX)$ that are  generated by an exceptional sequence, ordered by inclusion; and
    \item the  interval poset $[\idop, c]$, ordered by the absolute order.
           \end{itemize}
           \end{theorem}

Every hyperbolic cover of a Coxeter system is isomorphic to that Coxeter
system by \cite[Proposition~7.4]{BMW25}.
Therefore, it follows from Happel's result  \cite[Theorem 3.1]{Hap01} and from \cite[Theorem 5.1]{Br07} that 
the next theorem is a consequence of 
\cite[Theorem 1.2]{HK16} and 
Theorem~\ref{thm:mainRep}.

\begin{theorem} \label{main_generalCor}
Let $\mathcal C$ be a hereditary connected ext-finite abelian $k-$category with a tilting object over an algebraically closed field $k$ of characteristic zero.
%which is not derived equivalent to $\COH(\XX)$ for a weighted projective line $\XX$ of tubular type. 
Let $\Phi$ be the associated 
  generalized root system, $\tilde{W}$  a hyperbolic cover 
  of the associated Weyl group $W$, and $c \in \tilde{W}$ a Coxeter transformation. Then there exists an order preserving bijection between
  
  \begin{itemize}
    \item the poset of thick subcategories of $\mathcal C$ that are generated by an exceptional sequence, ordered by inclusion; and
    \item  the interval poset $[\idop, c]$,
    ordered by the absolute order.
           \end{itemize}
\end{theorem}

As another application of  Theorem~\ref{thm:main} and \cite[Theorem 1.3]{BWY21} we
 obtain a characterization of the set of all distinguished bases of vanishing cycles of a simple elliptic and hyperbolic singularity (see Section \ref{subsec:singularity}) and hereby answering a question of Brieskorn \cite[Question 4]{Ebe19} in these cases.

\begin{theorem}\label{thm:dist_bases}
Let $f$ be a simple elliptic or hyperbolic singularity, $\Lambda^*$ the set of vanishing cycles of $f$, $M$ the corresponding Milnor lattice and $h_*$ the monodromy operator of $f$. Then the set of all distinguished bases of vanishing cycles of $f$ is given by the set
$$
\{ (\delta_1 , \ldots, \delta_n) \in (\Lambda^*)^n \mid \spanz(\delta_1, \ldots , \delta_n)=M, ~s_{\delta_1} \cdots s_{\delta_n}=h_* \}.
$$
\end{theorem}

As here the extended Weyl group of type $D_4^{(1,1)}$ is not appearing, the mistake found by Charly Schwabe in a previous version of \cite{BWY21} does not play a role.

On the way to prove Theorem~\ref{thm:main} we show some results on Coxeter groups, which are interesting by themselves.
For instance the following. In a finite Coxeter group of rank $n$ with parabolic subgroup $P$
of rank $n-1$ all the reflections $t \in T$ which fulfill the property that they generate together with $P$ the whole Coxeter group are conjugate by the elements in $P$ (see \cite[Theorem~1.5]{BW18}). Here we show this result for the Coxeter groups whose Coxeter diagram is a star and its parabolic subgroup obtained by removing the unique vertex of
degree at least three from the Coxeter diagram (see Corollary~\ref{Coro:conj}).

\subsection*{Structure of the paper}

In Section \ref{sec:Notation} we recall the necessary framework, which is the extended Weyl group, the corresponding root system and the Coxeter transformation, from \cite{BWY21}. Then in Section \ref{sec:NormalForm} the structure
of the extended Weyl groups of wild and domestic type is determined.
Section~\ref{Sec:HyperbolicCover} recalls the notion of a hyperbolic cover, and in Section~\ref{sec:ExtCoxGrpRootLattice}
we show that the hyperbolic covers of the extended Coxeter systems are precisely the extended Coxeter systems.
In Section~\ref{sec:CoxTrans} we will show that all the Coxeter transformations are conjugate and will recollect the reflection length of a Coxeter transformation from \cite{BMW25}. As a preparation of the main results we consider Coxeter systems of star type  in Section~\ref{sec:AuxProofMain}.
The results of the previous sections are used in 
Section~\ref{sec:ProofMain} to prove the main Theorems~\ref{thm:main} and \ref{thm:AllgAussagehyperCover} uniformly. The last section is dedicated to the applications of the main theorems.

\subsection*{Acknowledgments}
The authors thank Henning Krause for coming up with the interesting topic and its application in the representation theory of finite dimensional algebras, as well as Wolfgang Ebeling for the discussion on singularity theory. The first author acknowledges that this work was supported by the Deutsche Forschungsgemeinschaft (SFB-TRR 358/1 2023 - 491392403).

\section{Notation, terminology and basic facts}\label{sec:Notation}

In this section we introduce most of the notation that we use throughout this paper. In particular
we recall the definition of a generalized root system and define the extended Weyl groups. These definitions, which  can be partially found in \cite{STW16} or \cite{Sai85}, will be used throughout the paper.  Slightly unusually  we will call a graph a \defn{star}, if it is  a tree and if there is at most one vertex of degree larger than $2$ and all the other vertices have degree $1$ or $2$.

\subsection{The extended Coxeter-Dynkin diagram}

\begin{figure}[h!]
  \centering
  \begin{tikzpicture}[scale=3.4]

    \node (A6666) at (0,-4) [circle, draw, fill=black!50, inner sep=0pt, minimum width=4pt] {};
    \node (AA6666) at (0,-3.9) [] {\tiny{$(1,p_1)$}};
    \node (A666) at (0.5,-4) [circle, draw, fill=black!50, inner sep=0pt, minimum width=4pt] {};
    \node (AA666) at (0.5,-3.9) [] {\tiny{$(1,p_1-1)$}};
    \node (A66) at (0.75,-4) [] {$\ldots$};
    \node (A6) at (1,-4) [circle, draw, fill=black!50, inner sep=0pt, minimum width=4pt] {};
    \node (AA6) at (1,-3.9) [] {\tiny{$(1,2)$}};
    \node (B6) at (1.5,-4) [circle, draw, fill=black!50, inner sep=0pt, minimum width=4pt]{};
    \node (BB6) at (1.45,-3.9) [] {\tiny{$(1,1)$}};
    \node (C6) at (2,-4) [circle, draw, fill=black!50, inner sep=0pt, minimum width=4pt]{};
    \node (CC6) at (2,-4.1) [] {\tiny{$1$}};
    \node (D6) at (2.5,-4) [circle, draw, fill=black!50, inner sep=0pt, minimum width=4pt]{};
    \node (DD6) at (2.55,-3.9) [] {\tiny{$(r,1)$}};
    \node (E6) at (3,-4) [circle, draw, fill=black!50, inner sep=0pt, minimum width=4pt]{};
    \node (EE6) at (3,-3.9) [] {\tiny{$(r,2)$}};
    \node (E66) at (3.25,-4) [] {$\ldots$};
    \node (F6) at (3.5,-4) [circle, draw, fill=black!50, inner sep=0pt, minimum width=4pt]{};
    \node (FF6) at (3.5,-3.9) [] {\tiny{$(r,p_r-1)$}};
    \node (G6) at (4,-4) [circle, draw, fill=black!50, inner sep=0pt, minimum width=4pt]{};
    \node (GG6) at (4,-3.9) [] {\tiny{$(r,p_r)$}};
    % \node (I66) at (2,-3.5) [circle, fill=white, inner sep=0pt, minimum width=9pt]{};
    \node (I6) at (2,-3.5) [circle, draw, fill=black!50, inner sep=0pt, minimum width=4pt]{};
    \node (II6) at (2,-3.4) [] {\tiny{$1^*$}};
    \node (J6) at (2.35,-4.2) [circle, draw, fill=black!50, inner sep=0pt, minimum width=4pt]{};
    \node (JJ6) at (2.55,-4.15) []{\tiny{$(r-1,1)$}};
    \node (J66) at (2.7,-4.4) [circle, draw, fill=black!50, inner sep=0pt, minimum width=4pt]{};
    \node (JJ66) at (2.9,-4.35) []{\tiny{$(r-1,2)$}};
    \node (J66666) at (2.88,-4.5) []{$\ldots$};
    \node (J666) at (3.05,-4.6) [circle, draw, fill=black!50, inner sep=0pt, minimum width=4pt]{};
    \node (JJ666) at (3.4,-4.55) []{\tiny{$(r-1,p_{r-1}-1)$}};
    \node (J6666) at (3.4,-4.8) [circle, draw, fill=black!50, inner sep=0pt, minimum width=4pt]{};
    \node (JJ6666) at (3.75,-4.75) []{\tiny{$(r-1,p_{r-1})$}};
    \node (K6) at (1.65,-4.2) [circle, draw, fill=black!50, inner sep=0pt, minimum width=4pt]{};
    \node (KK6) at (1.55,-4.15) []{\tiny{$(2,1)$}};
    \node (K66) at (1.3,-4.4) [circle, draw, fill=black!50, inner sep=0pt, minimum width=4pt]{};
    \node (KK66) at (1.2,-4.35) []{\tiny{$(2,2)$}};
    \node (K66666) at (1.12,-4.5) []{\ldots};
    \node (K666) at (0.95,-4.6) [circle, draw, fill=black!50, inner sep=0pt, minimum width=4pt]{};
    \node (KK666) at (0.75,-4.55) []{\tiny{$(2,p_2-1)$}};
    \node (K6666) at (0.6,-4.8) [circle, draw, fill=black!50, inner sep=0pt, minimum width=4pt]{};
    \node (KK6666) at (0.4,-4.75) []{\tiny{$(2,p_2)$}};
    \node (L6) at (2,-4.5) []{\ldots};

    % Kanten

    \draw[-] (A666) to (A6666);
    \draw[-] (J666) to (J6666);
    \draw[-] (K666) to (K6666);
    \draw[-] (A6) to (B6);
    \draw[-] (B6) to (C6);
    \draw[-] (C6) to (D6);
    \draw[-] (D6) to (E6);
    %    \draw[-] (E6) to (F6);
    \draw[-] (F6) to (G6);
    %    \draw[-] (G6) to (H6);
    \draw[-] (I6) to (B6);
    \draw[-] (I6) to (D6);
    \draw[-] (I6) to (J6);
    \draw[-] (I6) to (K6);
    \draw[-] (J6) to (J66);
    \draw[-] (K6) to (K66);
    \draw[-] (C6) to (J6);
    \draw[-] (C6) to (K6);
    \draw[dashed] ([xshift=0.5]C6.north) to ([xshift=0.5]I6.south);
    \draw[dashed] ([xshift=-0.5]C6.north) to ([xshift=-0.5]I6.south);

  \end{tikzpicture}
  \caption{Extended Coxeter-Dynkin diagram}\label{def:GenCoxDiag}
\end{figure}

We start by introducing a diagram, the so-called \defn{extended Coxeter-Dynkin diagram}.
For a non-negative integer $r\in \NN_{0}$ and numbers $p_{i}\in \NN$ where $1 \leq i \leq r$ it is defined as 
given in Figure~\ref{def:GenCoxDiag}.

\subsection{The extended space}\label{sec:ExtendedSpace}

We attach analogously to the Tits representation for Coxeter groups (see \cite[Section 5.3]{Hum90}) a geometric datum to the extended Coxeter-Dynkin diagram.
This will lead to the definition of the extended Weyl group and of a Coxeter transformation.

Let $Q$ be the vertex set of the diagram in Figure \ref{def:GenCoxDiag}. Define $V$ to be the real vector space $V$ with basis $B:=\lbrace \alpha_{\nu}\mid \nu\in Q\rbrace$. We further define a symmetric bilinear form $(- \mid -)$ on $V$ by setting 

\begin{equation*}
  (\alpha_{\nu}\mid \alpha_{\omega}) =
  \begin{cases}
    2  & \nu,\omega \in Q \text{ are connected by a dotted double bound or }\nu=\omega, \\
    0  & \nu,\omega \in Q \text{ are disconnected,}\\          
    -1 & \nu,\omega \in Q \text{ are connected by a single edge},
  \end{cases}
\end{equation*}
and extending bilinearly to $V$. We call $(V,B, (-\mid -))$ an \defn{extended space}.

The following definition is a generalization of the notation of a root system in the theory of finite Coxeter systems and can be found in \cite[Section 1.2]{Sai85}.
As usual, we define for $\alpha \in V$ non-isotropic, that is $(\alpha\mid \alpha) \neq 0$,  the \defn{reflection}
$$s_{\alpha}(v):=v-\frac{2(\alpha \mid v)}{(\alpha \mid \alpha)}\alpha ~~\mbox{for all }~v\in  V.$$

\begin{definition}\label{def:root_system}
A non-empty subset $\Phi\subseteq V$ of non-isotropic vectors is called \defn{generalized root system} if the following properties are satisfied
\begin{enumerate}
\item[(a)] $\spanr(\Phi)= V$,
\item[(b)] %$(\alpha \mid \alpha)\neq 0$ for all $\alpha \in \Phi$,
%\item[(c)] 
$s_{\alpha}(\Phi) \subseteq \Phi$ for all $\alpha\in \Phi$, and
\item[(c)] $\Phi$ is \defn{crystallographic}, i.e.
$\frac{2(\alpha \mid \beta)}{(\beta \mid \beta)} \in \mathbb{Z}$ for all $\alpha,\beta \in \Phi$.
%$\frac{2(\alpha,\beta)}{(\alpha,\alpha)}\in \mathbbm{Z}$ for all $\alpha,\beta\in \Phi$.
\end{enumerate}

The elements in a generalized root system are called \defn{roots}.
A generalized root system $\Phi$ is called \defn{simply laced} if $(\alpha \mid \alpha)=2$ for all $\alpha\in \Phi$, it is called \defn{reduced} if $r\alpha\in \Phi$ ($r\in \RR$) implies $r=\pm 1$ for all $\alpha\in \Phi$ and it is called \defn{irreducible} if there do not exist generalized root systems $\Phi_{1},\Phi_{2}$ such that $\Phi=\Phi_{1} \cup \Phi_{2}$ as well as $\Phi_{1} \bot \Phi_{2}$.

We call a non-empty subset $\Psi$ of a generalized root system $\Phi$ a \defn{root subsystem} if $\Psi$ is a generalized root system in $\spanr(\Psi)$. 
%It therefore suffices that condition $(c)$ is satisfied.

Further two generalized root systems $\Phi_{1}$ and $\Phi_{2}$ are called \defn{isomorphic} if there exists a linear isometry between the corresponding ambient spaces that sends $\Phi_{1}$ to $\Phi_{2}$.
\end{definition}

Notice that in \cite[Section 2.1]{STW16} a simply laced generalized root system consists of more information, namely of a lattice, a symmetric bilinear form, a set of roots and an element called Coxeter transformation.
 
 A direct calculation yields the signature of the symmetric bilinear form $(-\mid -)$ in an 
extended space:

\begin{Lemma}\label{lem:signature}
Let $(V, B, (-\mid-))$ be an extended space.  The signature of $(-\mid -)$ is $(|B|-2,1,1)$, $(|B|-2,0,2)$ or $(|B|-1,0,1)$ where the first, the second and the third entries are the geometric dimensions of the positive, the negative and the  zero eigenvalues of the form,
respectively.
\end{Lemma}

\begin{remark}
$\phantom{0}$
\begin{itemize}
\item[(a)] Notice that not all the possible triples are the signature of an extended space (see 
Proposition~\ref{def:basic} (a)).
\item[(b)]  Lemma \ref{lem:signature} can also be deduced for $r\geq 3$ in terms of the tilting theory. In fact, it is a combination of \cite[Proposition 18.8]{Len99} and \cite[Section 10]{HTT07}. For the so-called $T_{p,q,r}$ cases it is also proven in \cite[Chapter 5.11]{Ebe07}.
\end{itemize}
\end{remark}

\subsection{The extended Weyl system}

\begin{propdef} \label{def:basic}
Let $(V, B, (-\mid -))$ be an extended space.
\begin{enumerate}
\item[(a)] The group $W:=\langle s_{\alpha}\mid \alpha\in B\rangle$ is called \defn{extended Weyl group}. We call it \defn{wild},  \defn{tubular} or \defn{domestic} if the signature of $(-\mid -)$ is $(|B|-2,1,1)$, $(|B|-2,0,2)$ or $(|B|-1,0,1)$, respectively. 
Further we denote by $R$ the radical of the form $(-\mid -)$.
It is easy to see that the signatures that occur are exactly the following:
\begin{align*}
\text{tubular: }&(m,0,2) \text{ for } m=4,6,7,8\\
\text{domestic: }&(m,0,1) \text{ for } m\in \NN \\
\text{wild: }&(m,1,1) \text{ for } m\geq 6
\end{align*}

\item[(b)] 
If $W$ is of  domestic or wild type, then dim $R = 1$ and $R = \spanr(a)$ where
$a:=\alpha_{1^{*}}-\alpha_{1}$.
If $W$ is of tubular type, then $\spanr(a) \subset R$ and dim$_\RR(R) = 2$.

\item[(c)] The set $S:=\lbrace s_{\alpha}\mid \alpha \in B \rbrace$ is called \defn{simple system} and its elements  \defn{simple reflections}. We call $(W,S)$ an \defn{extended Weyl system}.
The set $T:=\bigcup_{w\in W}wSw^{-1}$ is  the \defn{set of reflections} for $(W,S)$.

\item[(d)] Let $\Phi\subseteq V$ be the minimal set that contains $B$ and is closed under the action of $W$, i.e. $w(\beta)\in \Phi$ for all $w\in W$ and $\beta \in \Phi$. Then $\Phi$ is a reduced 
and irreducible generalized root system  and $\Phi=W(B)$. We call $\Phi$ \defn{extended root system} and the elements of $B$ are called \defn{simple roots}.
\item[(e)] Let $\overline{\Phi}\subseteq V$ be the minimal set that contains $\overline{B}:=B\setminus \lbrace \alpha_{1^{*}} \rbrace$ and is closed under the action of $\overline{W}=\langle s_{\alpha}\mid \alpha \in \overline{B}  \rangle$.
Then $\overline{\Phi}$ is a 
simply laced 
generalized root system, which is irreducible and reduced, and it holds $\overline{\Phi}=\overline{W}(\overline{B})$. We call the set $\overline{\Phi}$ the \defn{projected root system}. By construction,
$(\overline{W},\overline{S})$ is a Coxeter system, where $\overline{S}=S\setminus \lbrace s_{1^{*}}\rbrace$. We call it the \defn{projected Coxeter system}. 

\item[(f)] As $(\overline{W}, \overline{S})$ is a Coxeter system whose Coxeter diagram is a  tree and whose simple roots are all of the same length, the action of $\overline{W} \subseteq W$
is transitive on $\overline{\Phi}$.  This also implies that all the roots in $ \overline{\Phi}$, and therefore also in $ \Phi$, are of the same length.

\item[(g)] The elements $c:=\left(\prod_{\alpha \in B\setminus \lbrace \alpha_{1},\alpha_{1^{*}}\rbrace}s_{\alpha}\right)\cdot s_{\alpha_{1}}s_{\alpha_{1^{*}}}$ are called \defn{Coxeter transformations} where we take the first $|B|-2$ factors 
in an arbitrary order.
\end{enumerate}
\end{propdef}

\begin{remark}\label{Rem:RootSystem}
$\phantom{4}$
\begin{itemize}
\item[(a)] 
Extended Weyl groups of tubular type are also called \defn{tubular elliptic} (see \cite{BWY21}).
\item[(b)]
By definition the root system $\overline{\Phi}$ is a root subsystem of $\Phi$, and $\overline{W}$ is a subgroup of $W$. Let $p$ be the natural projection of $V$ onto $\overline{V}:= V/R$. Then $p(\Phi)$ is 
isomorphic to $\overline{\Phi}$, and $p$ restricted to $\overline{\Phi}$ is injective. If it is convenient we will abbreviate $p(v)$ by $\overline{v}$ for $v \in V$.  In particular, we have $\alpha = \overline{\alpha}$ for $\alpha \in \overline{\Phi}$ in our setting.
In the following we will also identify $\overline{V}$ with the $\RR$-span of $\overline{B}$ in $V$.
\item[(c)] Notice that the extended Weyl groups $W$ of domestic type are precisely the affine simply laced irreducible Coxeter groups. We claim: if $W$ is an extended Weyl group of domestic type then there is $Q \subseteq T$ such that $(W,Q)$ is an affine Coxeter system of type $\widetilde{X}$ where $X$ is the type of the Coxeter system $(\overline{W}, \overline{S})$. Further the set of reflections in the Coxeter system $(W,Q)$ coincides with $T$.

This can be seen as follows. If $W$ is of domestic type then $(-\mid -)$ restricted to $\spanr(\overline{B})$ is positive definite, let us say $(\overline{W}, \overline{S})$ is of type $X$. Let $\widetilde{\alpha}$ be the highest root in the positive subsystem of type $X$ containing $\overline{B}$. By \ref{def:basic} (f) there is $w \in \overline{W}$ that maps $\alpha_1$ onto $-\widetilde{\alpha}$. This shows that $s_{-\widetilde{\alpha}+a} \in W$.
Then, as $\overline{S}$ generates $\overline{W}$, the group $W$ is generated by $Q:= \overline{S} \cup \{s_{-\widetilde{\alpha}+a}\} \subseteq T$. As the diagram for $Q$ is an affine Coxeter diagram, $(W,Q)$ is an affine Coxeter group with simple system $Q$ (see \cite[Section~6.5]{Hum90}).
Let $\widetilde{T}$ be the set of reflections in $(W,Q)$. Then 
$\widetilde{T} = \bigcup_{w\in W}wQw^{-1} \subseteq T$.
As $w^{-1}(\widetilde{\alpha}) = \alpha_1 $ it is
$s_{\alpha_{1^{*}}}$ contained in $\widetilde{T}$ and therefore we also have
$T \subseteq \widetilde{T}$.

On the other hand, if $(W,Q)$ is an affine simply laced irreducible Coxeter system where $Q$ is chosen as in the last paragraph, then we obtain an extended Weyl system $(W,S)$ by a similiar procedure as just described, where we choose $\alpha_1$ as the root related to the unique vertex of degree three, if $W$ is not of type $\widetilde{A}_n$, and else $\alpha_1$ can be chosen as any vertex of the Coxeter diagram of $(W,Q)$.
\end{itemize}
\end{remark}

\subsection{The Hurwitz action}
\medskip
\begin{definition}\label{def:hurwitz}
Let $G$ be an arbitrary group and $T \subseteq G$ a subset which is closed under conjugation. The braid group on $n$ strands, that is, the group
$$
\mathcal{B}_{n}:=\langle \sigma_{1}, \ldots , \sigma_{n-1} \mid \sigma_{i}\sigma_{j}=\sigma_{j}\sigma_{i} \text{ for }|i-j|>1, ~\sigma_{i}\sigma_{i+1}\sigma_{i}=\sigma_{i+1}\sigma_{i}\sigma_{i+1} \text{ for } 1\leq i \leq n-2 \rangle
$$ 
acts on $T^{n}$ as follows:
\begin{align*}
  \sigma_i (g_1 ,\ldots , g_n )      & = (g_1 ,\ldots , g_{i-1} , \hspace*{5pt} \phantom
  {g_i}g_{i+1}\phantom{g_i^{-1}}, \hspace*{5pt} g_{i+1}^{-1}g_ig_{i+1}, \hspace*{5pt} g_{i+2} ,
  \ldots , g_n),                                                                                      \\
  \sigma_i^{-1} (g_1 ,\ldots , g_n ) & = (g_1 ,\ldots , g_{i-1} , \hspace*{5pt} g_i g_{i+1} g_i^{-1},
  \hspace*{5pt} \phantom{g_{i+1}}g_i\phantom{g_{i+1}^{-1}}, \hspace*{5pt} g_{i+2} ,
  \ldots , g_n).
\end{align*}
This action is called \defn{Hurwitz action}.
\end{definition}
In this paper, we consider the Hurwitz action on $T^{|S|}$, where $W$ is an extended Weyl group, $S$ its simple system and $T$ its set of reflections.

\section{A normal form for an extended Weyl group} \label{sec:NormalForm}

The structure of the extended Weyl groups of tubular type, that is of the \defn{elliptic Weyl groups}, has been determined  in \cite{BWY21}. In this section we study the extended Weyl systems $(W,S)$ of wild and domestic type. First we determine the related generalized root system $\Phi$. Then we describe the elements of $W$ in terms of the root lattice of $\Phi$ by using the so called Eichler-Siegel transformation (see \cite{BWY21, Sai85}). This enables us to investigate the structure of the group $W$ and to introduce a normal form for the elements of $W$. Some of these results can also be found in \cite{STW16}.

Recall that for extended Weyl groups of domestic or wild type we have  $\dim(R) = 1$ and $R = \spanr(a)$, where $a =\alpha_{1^{*}}-\alpha_{1}.$

\subsection{The extended root system and its reflections}

\begin{Lemma}\label{lem:root_system_I2}
The following hold:
\begin{itemize}
\item[(a)] The smallest generalized root system that contains $\lbrace \alpha_{1},\alpha_{1^{*}}\rbrace$ is $\Phi':=\lbrace \pm \alpha_{1}+ka\mid k\in \ZZ\rbrace$.
\item[(b)] The generalized root system attached to $W$ is $\Phi=\lbrace \overline{\alpha}+ka\mid \overline{\alpha}\in \overline{\Phi},k\in \ZZ\rbrace$.
\end{itemize}
\end{Lemma}

\begin{proof}
The group $\langle s_{\alpha_{1}},s_{\alpha_{1^{*}}}\rangle$ is the dihedral group of type $\widetilde{A_{1}}$, thus the corresponding root system is $\Phi'=\lbrace \pm \alpha_{1}+ka\mid k\in \ZZ\rbrace$ by \cite[Proposition 6.3]{Kac90}. This shows (a).

Set $\Lambda = \lbrace \overline{\alpha}+ka\mid \overline{\alpha}\in \overline{\Phi},k\in \ZZ\rbrace$. Next we prove that $\Phi= \Lambda$.
Therefore observe that $B \subseteq \Lambda$. Further it is easy to check that $W(\Lambda) \subseteq \Lambda$. Therefore \ref{def:basic} (d) implies that $\Phi = W(B) \subseteq W(\Lambda) \subseteq \Lambda$. By (a) we have
$\lbrace  \pm \alpha_{1}+ka\mid k\in \ZZ\rbrace \subseteq \Phi$. The fact that $\overline{W} \subseteq W$ is transitive on $\overline{\Phi}$, see \ref{def:basic} (f),  yields $\Lambda \subseteq \Phi$, and we get (b).
\end{proof}

\begin{Lemma}\label{lem:Reflections}
Let $\alpha \in \Phi$. Then the following holds:
\begin{itemize}
\item[(a)] The reflection $s_\alpha$ is  uniquely determined
by $\Mov(s_\alpha) := (s_\alpha - \idop) (V)$,
and  it holds $\Mov(s_{\alpha}) = \spanr(\alpha)$.
\item[(b)] Let $g \in O(V,(-\mid -)) = \{g \in \GL(V) \mid (g(u)\mid  g(v)) = (u \mid v)~\text{for all $u,v$ } \in V\}$.
Then $g s_\alpha g^{-1} = s_{g(\alpha)}$.  In particular, $w s_\alpha w^{-1} = s_{w(\alpha)}$ for all $w \in W$.
\end{itemize}
\end{Lemma}
\begin{proof}
As $\alpha\notin R$ assertion (a) follows directly from the definition of $s_{\alpha}$. To prove (b) let $v \in V$ and $u = g^{-1}(v)$, and calculate
$$g s_\alpha g^{-1}(v) - v = g(s_\alpha(u) - u) \in g( \Mov (\alpha)) = \spanr (g(\alpha)),$$
which implies the assertion by (a).
\end{proof}

The following nice fact is a well-known property of reflections. For completeness we present a proof.

\begin{Lemma}\label{lem:orth_complement_lin_ind_roots}
Let $\beta_{1},\ldots,\beta_{m}\in \Phi$ be linearly independent and $v\in V$, then $s_{\beta_{1}}\cdots s_{\beta_{m}}(v)=v$ if and only if $s_{\beta_{i}}(v)=v$ for all $1\leq i \leq m$.
\end{Lemma}

\begin{proof}
Applying $s_{\beta_1}$ to the equality $s_{\beta_{1}}\cdots s_{\beta_{m}}(v)=v$ yields 
$s_{\beta_{2}}\cdots s_{\beta_{m}}(v)= s_{\beta_1}(v) = v - (v \mid \beta_1) \beta_1.$ Therefore, 
\[-(v \mid \beta_{1})\beta_{1} = s_{\beta_{2}}\cdots s_{\beta_{m}}(v)-v   \in \spanr(\beta_{2},\ldots,\beta_{m}),\]
and the linear independence of $\lbrace \beta_{1},\ldots,\beta_{m}\rbrace$ implies $(\beta_{1} \mid v)=0$. The latter is equivalent to $s_{\beta_{1}}(v)=v$. By induction we get that $s_{\beta_{i}}(v)=v$ for all $2\leq i \leq m$.
\end{proof}

\subsection{The structure of an extended Weyl system $(W,S)$}\label{subsec:StructureW}

The next definition due to Saito is helpful for the understanding of $W$
(see also \cite{BWY21}).

\begin{definition}[{\cite[(1.14) Definition 1]{Sai85}}] \label{def:eichler}
It is 
\[
E:V\otimes_{\RR}\overline{V}\rightarrow \text{End}(V),~\sum_{i}u_{i}\otimes \overline{v}_{i}\mapsto \left[x\mapsto x-\sum_{i}(v_{i} \mid x)u_{i}\right],
\]
the so called \defn{Eichler-Siegel map}.
\end{definition}

We define a binary operation $\circ$ on $V\otimes_{\RR}\overline{V}$ by setting for $x_1,x_2 \in V\otimes_{\RR}\overline{V}$:
\[x_1 \circ x_2 = x_1 + x_2 - I(x_1,x_2)\]
where
\begin{align*}
I:\left(V\otimes_{\RR}\overline{V}\right) \times \left(V\otimes_{\RR}\overline{V}\right) \rightarrow V\otimes_{\RR}\overline{V},~(x_{1},x_{2})\mapsto I(x_{1},x_{2})
\end{align*} 
with
\[I(x_{1},x_{2}):=\sum_{i_{1},i_{2}}u_{1i_{1}}\otimes(v_{1i_{1}} \mid u_{2i_{2}})\overline{v}_{2i_{2}}\]
for 
$$x_{j}=\sum_{i_{j}}u_{ji_{j}}\otimes \overline{v}_{ji_{j}}\in V\otimes_{\RR}\overline{V},~(j=1,2).$$

The operation $\circ$ yields a semi-group structure on $V\otimes_{\RR}\overline{V}$. We obtain directly:

\begin{proposition}[{\cite[1.14]{Sai85}}] \label{prop:properties_Eichler}
$\phantom{4}$
Let $x_1, x_2 \in V\otimes_{\RR}\overline{V}$.
\begin{itemize}
\item[(a)] The map $E$ is injective. It is bijective if and only if $R=0$.
\item[(b)] The map $E$ is a homomorphism of semi-groups, that is $E(x_1 \circ x_2)=E(x_1)E(x_2)$.
\item[(c)] For a non-isotropic $v\in V$, the reflection $s_{v}$ is given by $s_{v}=E(v\otimes\overline{v})$.
\item[(d)] The inverse of the Eichler-Siegel map 
\[E^{-1}:W\rightarrow V\otimes_{\RR}\overline{V}\]
on $W$ is well defined.
\item[(e)] The subspace $R\otimes_{\RR} \overline{V}$ is closed under $\circ$,
and $\circ$ coincides on $(V\otimes_{\RR} \overline{V}) \times (R \otimes_{\RR} \overline{V})$  with the additive structure of $V \otimes_{\RR} \overline{V}$.
\item[(f)] Let $r \in \RR$ and $v \in V$ non-isotropic. Then $$E((v+ ra) \otimes \overline{v}) = E(v\otimes \overline{v})E(ra \otimes \overline{v}) = s_v E(ra \otimes \overline{v}).$$
\end{itemize}
\end{proposition}

\begin{definition}\label{def:Lattice}
Let $L_a := \spanz(B_a)$ where $B_a:=
\{a\otimes \overline{\alpha}~|~\alpha \in \overline{B}\}$, and set $L := \spanz(B)$ and $\overline{L} := \spanz(\overline{B})$.
\end{definition}

\begin{remark}
    Notice that $\overline{L}$ is the root lattice of the Coxeter system 
    $(\overline{W}, \overline{S})$.
\end{remark}

We get as an immediate consequence of Proposition~\ref{prop:properties_Eichler} (b) and (c) the following.

\begin{Lemma}\label{lem:Generation}
$\phantom{0}$
\begin{itemize}
\item[(a)] The three lattices $L/(L \cap R)$, $\overline{L}$ and $L_a$ are isomorphic;
\item[(b)] $W \subseteq E(L \otimes_\ZZ \overline{L}) \subseteq E(V \otimes_\RR \overline{V})$.
\end{itemize}
\end{Lemma}

\begin{Lemma}\label{lem:Translation}
Let $\alpha \in  \Phi$ and $k \in \ZZ$.
Then the following hold.
\begin{itemize}
\item[(a)] $s_{\alpha} s_{\alpha +ka} = E(ka \otimes \overline{\alpha})$;
\item[(b)] $w E(ka \otimes  \overline{\alpha}) w^{-1} = E(ka \otimes
\overline{w(\alpha)})$ for all $w \in W$;
\item[(c)] $E(L_a)$ is an abelian group, which is normalized by $W$.
\end{itemize}
\end{Lemma}
\begin{proof}
Assertion (a) follows from Proposition~\ref{prop:properties_Eichler} (c) and (f), and (b) is a consequence of Lemma~\ref{lem:Reflections} (b) and of (a).

By Proposition~\ref{prop:properties_Eichler} (b) and (e) $E(L_a)$ is an abelian group, and 
it is normalized by $W$ by (b).
\end{proof}

We can make a first statement on the structure of $W$. 

\begin{Lemma} \label{lem:nat_proj}
The natural projection $p: V\rightarrow \overline{V}$ induces an epimorphism $\rho:W\rightarrow \overline{W}$. 
\end{Lemma}
\begin{proof} 
As $R$ is the radical of the bilinear form $(-\mid-)$ on $V$, the group $W$ acts trivially on $R$. 
Further 
it also acts on $\overline{V}$. This action is given by the map  $\rho: s_\alpha \mapsto s_{\overline{\alpha}}$
for $\alpha \in\Phi$ where $s_{\overline{\alpha}}$ is the  reflection of $\overline{V}$ with $\Mov(s_{\overline{\alpha}}) = \spanr (\overline{\alpha})$.

As $\overline{W}$ acts faithfully on the complement $\Span_{\RR}(\{ b \mid b \in \overline{B} \})$
to $R$ in $V$, it also acts faithfully on $\overline{V}$ and therefore $\ker (\rho) \cap \overline{W} = \{1\}$.
Since $W = \langle s_\alpha~|~ \alpha \in B\rangle$ and since $E(a \otimes \overline{\alpha}_1) = 
  s_{\alpha_{1}}s_{\alpha_{1^{*}}} \in \ker(\rho)$, it follows by the Dedekind identity that $\ker(\rho)$ 
equals $ \langle\langle E(a \otimes \overline{\alpha}_1) \rangle \rangle$, the normal closure $M$ of $E(a \otimes \overline{\alpha}_1)$ in $W$, and that $W = \overline{W} \ltimes M$. 
This yields  that we may  identify $\rho(W)$ and $\overline{W}$ via the isomorphism $\rho_{\mid \overline{W}}$, and thereby get the assertion.
\end{proof}

\medskip
\begin{proposition}[{\cite[1.15]{Sai85}, \cite[Theorem 3.5]{STW16}}]\label{thm:isom_normalform}
The short exact sequence 
\begin{align*}
0\longrightarrow  E^{-1}(W) \cap (R\otimes_{\RR} (\overline{V})) \overset{E}{\longrightarrow} W \overset{\rho}{\longrightarrow} \overline{W} \longrightarrow 1 
\end{align*}
splits. Further $ E^{-1}(W) \cap (R\otimes_{\RR} (\overline{V})) = L_a \cong \overline{L}$ and 
\begin{itemize}
\item[(a)] $W = \overline{W} \ltimes E(L_a) \cong \overline{W} \ltimes \overline{L}$;
\item[(b)] the action of $\overline{W}$ on $ E(L_a) \cong \overline{L}$ is given in Lemma~\ref{lem:Translation} (b). 
In particular, $\overline{W}$ acts on $E(L_a) \cong \overline{L}$ as 
$\overline{W}$ acts on its root lattice.
\end{itemize}
\end{proposition}

\begin{proof} 
Let $T_{R} : = E^{-1}(W) \cap (R\otimes_{\RR} (\overline{V}))$.
The following diagram is commutative

\begin{tikzcd}
\centering
0\arrow{r} & T_{R} \arrow{r}{E} \arrow[d,hook]& W\arrow{r}{\rho} \arrow[d, hook,"E^{-1}"] &\overline{W}\arrow[d,hook,"E^{-1}"] \arrow{r}& 1\\
0\arrow{r} & R\otimes_{\RR} (\overline{V})\arrow{r}& V\otimes_{\RR}(\overline{V})\arrow{r} &(V/R)\otimes_{\RR}(\overline{V})\arrow{r}& 0,
\end{tikzcd}
\medskip\\
and the second row is exact, which implies the exactness of the first row. In particular, this also implies that $E(T_{R}) = \ker(\rho)$. The first sequence splits because of Lemma~\ref{lem:nat_proj}. By the same lemma, $\ker(\rho)$ is $\langle\langle E(a \otimes \overline{\alpha}_1) \rangle\rangle$, the normal closure of $E(a \otimes \overline{\alpha}_1)$ in $W$. Since $\overline{W}$ acts transitively on $\overline{\Phi}$ we get that $E(a \otimes \overline{\alpha}) \in  \langle\langle E(a \otimes \overline{\alpha}_1)\rangle\rangle$ for every $\alpha \in \overline{B}$. Therefore $E(L_a) \subseteq \ker(\rho)$, and $\ker(\rho) = E(L_a)$ follows with Lemma~\ref{lem:Translation} (c).
\end{proof}

\subsection{A normal form for the elements in $W$}\label{subsec:NormalForm}

Proposition~\ref{thm:isom_normalform} yields a normal form for the elements in $W$. Observe if $w \in W$, then $w = \overline{w} P(w)$ where $\overline{w} = \rho(w)$ and where $P$ is the projection of $W$ onto $\ker(\rho)$. It is $P(w) = E(a\otimes \overline{\beta})$ for some $\beta \in \overline{L}$ by Proposition~\ref{thm:isom_normalform}.

\begin{definition}\label{def:normalform}
Define a map $\TR: W \rightarrow \overline{L}$ by setting for $w \in W$ 
$$\TR(w) = \beta~\mbox{if}~\beta \in \overline{L}~\mbox{is such that}~ P(w) = E(a \otimes \overline{\beta}).$$
We call the pair $(\overline{w},\TR(w))$ the \defn{normal form of $w$} and $\TR(w)$ the \defn{translation vector} of $w$. 
\end{definition}

\begin{Lemma}\label{lem:trv_for_ref}
The following holds.

\begin{itemize}
\item[(a)] Let $w\in W$ such that $\TR(w)=\sum_{\beta \in \overline{B}}m_{\beta }\cdot \beta$ where $m_{\beta} \in \ZZ$.
Then 
\[w=\overline{w}\prod_{\beta \in \overline{B}}(s_{\beta}s_{\beta+a})^{m_{\beta}}.\]

\item[(b)] Let $\alpha=\overline{\alpha}+ka\in \Phi$ where $\overline{\alpha}\in \overline{\Phi}$ and $k\in \ZZ$. Then $\TR(s_{\alpha})=k\overline{\alpha}$.

\item[(c)] For $\gamma \in \overline{B}$, $m_{\gamma}\in \ZZ$ and $y\in \overline{W}$ we have
 \[ \TR (y^{-1}\prod_{\gamma\in \overline{B}}(s_{\gamma}s_{\gamma+a})^{m_{\gamma}}y)=\sum_{\gamma \in \overline{B}}m_{\gamma} \cdot y^{-1}(\gamma)= y^{-1}\Big(\TR \Big(\prod_{\gamma\in \overline{B}}(s_{\gamma}s_{\gamma+a})^{m_{\gamma}}\Big) \Big).
 \]
\end{itemize}
\end{Lemma}

\begin{proof}
$\phantom{4}$
Lemma~\ref{lem:Translation} (a) and Proposition~\ref{prop:properties_Eichler} (e) yield (a), while Proposition~\ref{prop:properties_Eichler} (f) yields assertion (b). The third assertion is a direct consequence of Lemma~\ref{lem:Translation}.
\end{proof}

\begin{Lemma}\label{lem:translation_vector}
The translation vector satisfies the following properties:
  \begin{itemize}
    \item[(a)] $\TR (s_{1^{*}})=\alpha_{1}$;
    \item[(b)] $\TR (s)=0$ for all $s\in S\setminus \lbrace s_{1^{*}}\rbrace$;
    \item[(c)] $\TR (xy)=y^{-1}\TR (x)+\TR (y)$ for all $x,y \in W$.
  \end{itemize}
In particular, the translation vector of an element of $W$ is uniquely determined by the properties $(a)-(c)$.
\end{Lemma}

\begin{proof}
Lemma \ref{lem:trv_for_ref} (b) yields $\TR(s_{\alpha_{1^{*}}})=\TR(s_{\alpha_{1}+a})=\alpha_{1}$, which is (a). Assertion (b) follows from Proposition~\ref{prop:properties_Eichler} (c) and (f). Next we prove (c). Let $x,y \in W$ and $\beta_x = \TR(x), ~\beta_y = \TR(y)$.
Then by Lemma~\ref{lem:Translation} (b)
$$xy = \overline{x}E(a\otimes \beta_x)\overline{y}E(a\otimes \beta_y) = \overline{x} \overline{y} \overline{y}^{-1} E(a\otimes \beta_x) \overline{y} E(a\otimes \beta_y)
= \overline{x} \overline{y}E(a\otimes (\overline{y}^{-1}(\beta_x))+ \beta_y),$$
which yields (c).
\end{proof}

\begin{corollary}\label{cor:normalform_prop}
The elements in $T$ have normal form $(s_{\alpha},k \alpha)$ where  $k$ is any element in $\ZZ$  and $\alpha$ any root  in $\overline{\Phi}$.
\end{corollary}
\begin{proof}
If $t \in T$, then by Lemma~\ref{lem:root_system_I2} (b) there is a root $\alpha \in \overline{\Phi}$ and $k \in \ZZ$ such that $t = s_{\alpha +ka}$. By 
Lemma~\ref{lem:trv_for_ref} the normal form of $t$ is $(s_{\alpha},k \alpha)$.

On the other hand Lemma~\ref{lem:root_system_I2} (b) implies that for every $k\in \ZZ$ and $\alpha \in \overline{\Phi}$ the tuple $(s_{\alpha},k \alpha)$ is the normal form of some reflection.
\end{proof}

The following generalizes the respective property for affine simply laced Coxeter groups (see \cite[Lemma 2.11]{PW17a}).
Our proof is  a generalization of the one given in \cite{PW17a}. Therefore we leave out the details.

\begin{Lemma}\label{lem:zero_tran}
Let $\beta_{1},\ldots,\beta_{n}\in \Phi$ such that $\overline{\beta_{1}},\ldots,\overline{\beta_{n}}\in \overline{\Phi}$ are linearly independent. Then \linebreak $\TR(s_{\beta_{1}} \cdots s_{\beta_{n}})=0$ if and only if $\TR(s_{\beta_{1}})=\ldots = \TR(s_{\beta_{n}})=0$.
\end{Lemma}
\begin{proof}
It is easy to prove that (see \cite[Lemma 2.11]{PW17a})
\begin{align*}
\TR(s_{\beta_{1}}\cdots s_{\beta_{n}})=\sum_{i=0}^{n-2}s_{\overline{\beta_{n}}}\cdots s_{\overline{\beta_{n-i}}}\left(\TR \left( s_{\beta_{n-i-1}}\right)\right)+\TR \left(s_{\beta_{n}}\right). \label{eq1}
\end{align*}
Using the previous formula in an easy induction on the number of factors, the assertion follows.
\end{proof}

\section{The hyperbolic cover of an extended Weyl system}\label{Sec:HyperbolicCover}

In this section we construct the hyperbolic cover of an extended Weyl system $(W,S)$. First, we will introduce it for the groups of domestic or wild type, and then we will recall the construction for those of tubular type (see \cite{Sai85, BMW25}). Recall that $(V,B,(-\mid-))$ is the extended space for $(W,S)$ (see Section~\ref{sec:ExtendedSpace}).

\subsection{The hyperbolic cover of a domestic or wild type}\label{}
The extended Weyl group $W$ of domestic or wild type  also acts faithfully  on a non-degenerate space $\tilde{V}$. Recall that $\overline{V}$ is the subspace of $V$ that is generated by 
$\overline{B} = B\setminus{\{\alpha_{1^*}\}}$. Let $B = \{\alpha_1, \cdots , \alpha_n\}$.

The space $\tilde{V}$ is constructed as follows:
extend the $\RR$-space $V$ to the $\RR$-space  
$$\tilde{V} = V\oplus \RR a',~\mbox{and  set}~(\overline{V} \mid a') = 0~\mbox{and}~
(a\mid a') = 1.$$  Then the signature of $\tilde{V}$ is $(n+2, 1,0)$ and $(n+1, 2,0)$ in the domestic and the wild cases, respectively. The root system  $\Phi$ is contained in $V \subset \tilde{V}$, and we can consider the 
reflections $\tilde{s}_i$ of $\tilde{V}$ with respect to the root $\alpha_i \in \tilde{V}$ for $i \in \{1, \ldots ,n, 1^*\}$. 
We define the group generated by the $\tilde{s}_i$ 
$$\tilde{W}:= \langle s \mid s \in \tilde{S}\rangle ~\mbox{where}~
\tilde{S}:= \{\tilde{s}_1 , \ldots, \tilde{s}_n, \tilde{s}_{1^*} \},$$
and call the pair $(\tilde{W}, \tilde{S})$ \defn{hyperbolic cover} of $(W,S)$. By abuse of language we also call $\tilde{W}$ \defn{hyperbolic cover} of $W$.

\begin{proposition}\label{Prop:hypcoverWild}
The group $\tilde{W}$ acts faithfully on $V$, and $\tilde{W}$ is isomorphic to $W$.
\end{proposition}
 \begin{proof}
Let $J$ be the Gram matrix of $(-\mid -)$ with respect to the basis $\tilde{B}:= B \cup 
\{a'\}$ of $\tilde{V}$. As the form is left invariant by the elements of $\tilde{W}$, it is $A^tJA = J$ for every element $w$ of $\tilde{W}$ where we read $w$ as a matrix $A$ over $\tilde{B}$. From this we conclude that $w(a') = a'$ for every $w \in \tilde{W}$. Thus, $\tilde{W}$ and $W$ are isomorphic and, as the action of $W$ on $V$ is
faithful, the action of $\tilde{W}$ on $V$ is also faithful.
 \end{proof}

\subsection{The hyperbolic cover of a tubular type}
 We recall the definition of a hyperbolic cover $(\tilde{W}, \tilde{S})$ of an extended  Weyl system $(W,S)$ of tubular
 type given in \cite{BMW25}. 
 
In \cite{BMW25} the labeling of the extended Coxeter-Dynkin diagram $\Gamma$
is slightly different than in this paper (there it is adapted 
to \cite{Bou02}).
There the labeling of $\Gamma$ is by the set 
 $\{0, \ldots , n, t^*\}$  where either $t= 2$ in type $D_4^{(1,1)}$ or $t = 4$ else,
 and $\Gamma$ has precisely two vertices of valency at least $3$, which  are labeled by $t$ and $t^*$ instead of $1$ and $1^*$, respectively. The subdiagrams on the set of vertices  labeled by
 $1, \ldots , n$ and $0, \ldots , n$
  are a finite and an affine Coxeter diagram $\Gamma_{\fin}$ and $\Gamma_b$, respectively.

  \begin{figure}
  \centering
  \begin{tikzpicture}[scale=1.9]

    \node (03) at (0.5,0.5) [] {$D_4^{(1,1)}$};
    \node (G3) at (2,1) [circle, draw, fill=black!50, inner sep=0pt, minimum width=4pt]{};
    \node (A3) at (1.5,0.5) [circle, draw, fill=black!50, inner sep=0pt, minimum width=4pt]{};
    \node (B3) at (2.5,0.5) [circle, draw, fill=black!50, inner sep=0pt, minimum width=4pt]{};
    \node (C3) at (2,0.5) [circle, draw, fill=black!50, inner sep=0pt, minimum width=4pt]{};
    \node (E3) at (1.6, 0.1) [circle, draw, fill=black!50, inner sep=0pt, minimum width=4pt]{};
    \node (F3) at (2.4, 0.1) [circle, draw, fill=black!50, inner sep=0pt, minimum width=4pt]{};

    % Knoten
    \node (04) at (0.5,-1) [] {$E_6^{(1,1)}$};
    \node (A4) at (1,-1) [circle, draw, fill=black!50, inner sep=0pt, minimum width=4pt] {};
    \node (B4) at (1.5,-1) [circle, draw, fill=black!50, inner sep=0pt, minimum width=4pt]{};
    \node (C4) at (2,-1) [circle, draw, fill=black!50, inner sep=0pt, minimum width=4pt]{};
    \node (D4) at (2.5,-1) [circle, draw, fill=black!50, inner sep=0pt, minimum width=4pt]{};
    \node (E4) at (3,-1) [circle, draw, fill=black!50, inner sep=0pt, minimum width=4pt]{};
    \node (F4) at (2.4,-1.4) [circle, draw, fill=black!50, inner sep=0pt, minimum width=4pt]{};
    \node (G4) at (2.8,-1.8) [circle, draw, fill=black!50, inner sep=0pt, minimum width=4pt]{};
    \node (H4) at (2,-0.5) [circle, draw, fill=black!50, inner sep=0pt, minimum width=4pt]{};

    \node (05) at (4.2,0.5) [] {$E_7^{(1,1)}$};
    \node (A5) at (4.7,0.5) [circle, draw, fill=black!50, inner sep=0pt, minimum width=4pt] {};
    \node (B5) at (5.2,0.5) [circle, draw, fill=black!50, inner sep=0pt, minimum width=4pt]{};
    \node (C5) at (5.7,0.5) [circle, draw, fill=black!50, inner sep=0pt, minimum width=4pt]{};
    \node (D5) at (6.2,0.5) [circle, draw, fill=black!50, inner sep=0pt, minimum width=4pt]{};
    \node (E5) at (6.7,0.5) [circle, draw, fill=black!50, inner sep=0pt, minimum width=4pt]{};
    \node (H5) at (6.6,0.1) [circle, draw, fill=black!50, inner sep=0pt, minimum width=4pt]{};
    \node (F5) at (7.2,0.5) [circle, draw, fill=black!50, inner sep=0pt, minimum width=4pt]{};
    \node (G5) at (7.7,0.5) [circle, draw, fill=black!50, inner sep=0pt, minimum width=4pt]{};
    \node (I5) at (6.2,1) [circle, draw, fill=black!50, inner sep=0pt, minimum width=4pt]{};

    \node (06) at (4.2,-1) [] {$E_8^{(1,1)}$};
    \node (A6) at (4.7,-1) [circle, draw, fill=black!50, inner sep=0pt, minimum width=4pt] {};
    \node (B6) at (5.2,-1) [circle, draw, fill=black!50, inner sep=0pt, minimum width=4pt]{};
    \node (C6) at (5.7,-1) [circle, draw, fill=black!50, inner sep=0pt, minimum width=4pt]{};
    \node (D6) at (6.2,-1) [circle, draw, fill=black!50, inner sep=0pt, minimum width=4pt]{};
    \node (E6) at (6.7,-1) [circle, draw, fill=black!50, inner sep=0pt, minimum width=4pt]{};
    \node (F6) at (7.2,-1) [circle, draw, fill=black!50, inner sep=0pt, minimum width=4pt]{};
    \node (G6) at (7.7,-1) [circle, draw, fill=black!50, inner sep=0pt, minimum width=4pt]{};
    \node (H6) at (8.2,-1) [circle, draw, fill=black!50, inner sep=0pt, minimum width=4pt]{};
    \node (I6) at (5.7,-0.5) [circle, draw, fill=black!50, inner sep=0pt, minimum width=4pt]{};
    \node (J6) at (6.1,-1.4) [circle, draw, fill=black!50, inner sep=0pt, minimum width=4pt]{};
    
    % Kanten

    \draw[-] (A3) to (C3);
    \draw[-] (C3) to (B3);
    \draw[-] (C3) to (E3);
    \draw[-] (C3) to (F3);
    \draw[-] (G3) to (A3);
    \draw[-] (G3) to (B3);
    \draw[-] (G3) to (E3);
    \draw[-] (G3) to (F3);
    \draw[dashed] ([xshift=0.5]C3.north) to ([xshift=0.5]G3.south);
    \draw[dashed] ([xshift=-0.5]C3.north) to ([xshift=-0.5]G3.south);

    \draw[-] (A4) to (B4);
    \draw[-] (B4) to (C4);
    \draw[-] (C4) to (D4);
    \draw[-] (D4) to (E4);
    \draw[-] (C4) to (F4);
    \draw[-] (F4) to (G4);
    \draw[dashed] ([xshift=0.5]C4.north) to ([xshift=0.5]H4.south);
    \draw[dashed] ([xshift=-0.5]C4.north) to ([xshift=-0.5]H4.south);
    \draw[-] (B4) to (H4);
    \draw[-] (D4) to (H4);
    \draw[-] (F4) to (H4);

    \draw[-] (A5) to (B5);
    \draw[-] (B5) to (C5);
    \draw[-] (C5) to (D5);
    \draw[-] (D5) to (E5);
    \draw[-] (E5) to (F5);
    \draw[-] (F5) to (G5);
    \draw[-] (D5) to (H5);
    \draw[-] (I5) to (C5);
    \draw[-] (I5) to (E5);
    \draw[-] (I5) to (H5);
    \draw[dashed] ([xshift=0.5]D5.north) to ([xshift=0.5]I5.south);
    \draw[dashed] ([xshift=-0.5]D5.north) to ([xshift=-0.5]I5.south);

    \draw[-] (A6) to (B6);
    \draw[-] (B6) to (C6);
    \draw[-] (C6) to (D6);
    \draw[-] (D6) to (E6);
    \draw[-] (E6) to (F6);
    \draw[-] (F6) to (G6);
    \draw[-] (G6) to (H6);
    \draw[-] (I6) to (B6);
    \draw[-] (I6) to (D6);
    \draw[-] (I6) to (J6);
    \draw[-] (C6) to (J6);
    \draw[dashed] ([xshift=0.5]C6.north) to ([xshift=0.5]I6.south);
    \draw[dashed] ([xshift=-0.5]C6.north) to ([xshift=-0.5]I6.south);

  \end{tikzpicture}
  \caption{Elliptic Dynkin diagrams for the tubular elliptic root systems} \label{fig:EllipticDynkin}
\end{figure}

 The basis $B$ of $V$ is  $B= \{\alpha_0, \ldots , \alpha_n, \alpha_{t^*}\}$, and 
 the radical $R $ of the bilinear form $(-\mid -)$ is $2$-dimensional. 
 Then  $a= \alpha_{t^*} - \alpha_t$ and $b = \alpha_0 + \tilde{\alpha}$
 form  a basis of $R$ where $\tilde{\alpha}$ is the highest root in the finite root system given by $\{\alpha_1, \ldots , \alpha_n\}$, and $V = V_{\fin} \oplus R$, where $V_{\fin}$ is the $\RR$-span of $\alpha_1, \ldots , \alpha_n$.
 
 In \cite{BMW25} the projected root system $\overline{\Phi} \subseteq \Phi$ and 
 the projected Coxeter system $(\overline{W}, \overline{S})$ are denoted by $\Phi_b$ and $(W_b, S_b)$, respectively.
 Hence $\Phi_b$ is the subsystem of $\Phi$ generated by 
 $\alpha_0, \ldots , \alpha_n$, and $W_b$ the subgroup of $W$ generated by $s_0, \ldots, s_n$.
 Further it holds $ S_b = \{s_0, s_1, \cdots , s_n\} $, and $S_{\fin} = \{s_1, \cdots , s_n\} \subset S_b$, which  generates a finite Coxeter group. The Dynkin diagram $\Gamma_b$ for $(W_b, S_b)$ is a star.
 
 The space $V$ and the form $(-\mid -)$ are extended to 
 $$\tilde{V} = V\oplus \RR b',~\mbox{with}~(V_{\fin} \mid b') =  0 = (a \mid b')~\mbox{and}~
(b\mid b') = 1.$$ 
 Again it holds $\Phi \subset V \subset \tilde{V}$. We consider the 
reflections $\tilde{s}_i$ of $\tilde{V}$ with respect to the root $\alpha_i \in \tilde{V}$ for $i \in \{0, \ldots ,n, t^*\}$. 
The hyperbolic cover $(\tilde{W}, \tilde{S})$ of $(W,S)$ is defined by
$$\tilde{S}:= \{\tilde{s}_1 , \ldots, \tilde{s}_n, \tilde{s}_{t^*} \}~\mbox{and}~
\tilde{W}:= \langle s \mid s \in \tilde{S}\rangle .$$ 
In \cite{BMW25} we also extend $\tilde{V}$ by an element 
$a'$ as above and call the similarly constructed group $\hat{W}$.
The first two authors justify the name hyperbolic cover in \cite{BMW25} (see also \cite{ST97}).
 
\begin{proposition}[{\cite[Proposition 3.10]{BMW25}}]\label{StructureTildeW}
		The following hold.
	\begin{itemize}
		\item[(a)]	The group  $\tilde{W}$  is a non-split central extension  of $W$ by an infinite cyclic group.
		\item[(b)] The groups $\tilde{W}$ and $\hat{W}$  are isomorphic.
	\end{itemize} 
\end{proposition}

\section{ Extended Coxeter groups (by its root lattice)}\label{sec:ExtCoxGrpRootLattice}\label{Sec:CoxExtensionRootLattice}

In this section we recall extended Coxeter groups, which first appeared in \cite{Lek2}.

\subsection{The extended Coxeter system $(\mathcal{W}, \mathcal{S})$}\label{subsec:DefCoxExtRootLattice}

Let $(\overline{W}, \overline{S})$ be a crystallographic Coxeter system with a linear representation on $\overline{V}$, which is  equipped with a bilinear form $\langle -\mid -\rangle$. Further let $\overline{\Phi} \subseteq \overline{V}$ be a root system  for the Coxeter system and $L(\overline{\Phi})$ its root lattice.
 Then  $\overline{W}$ acts on $L(\overline{\Phi})$.
The semi-direct product
$$\mathcal{W} := \mathcal{W}(\overline{W}, \overline{S}):= \overline{W} \ltimes L(\overline{\Phi})$$
is the \defn{extended Coxeter group (by its root lattice)} (see \cite{Looi80, Lek, Lek2}).
We write the elements $w \in \mathcal{W}$ as pairs 
$w = (\overline{w},\lambda) $ where $\overline{w}$ is in $\overline{W}$ and $\lambda$ in $L(\overline{\Phi})$.
Similarly as in Section~\ref{sec:NormalForm} we consider a map $\TR: \mathcal{W} \rightarrow L(\overline{\Phi})$ sending $w =(\overline{w},\lambda) $ to $\lambda$.

We further assume that the Coxeter diagram $\overline{\Gamma}$ of the system $(\overline{W}, \overline{S})$ is a star, and 
that $\langle \alpha\mid \alpha \rangle = 2$ for all $\alpha$ in $\overline{\Phi}$.
Then we say that  $\mathcal{W}$ is of \defn{ star type}.

Label the vertices of $\overline{\Gamma}$ by $1$ to $n$, such that $1$ is the label of a vertex of maximal valency.
Further let $\alpha_i$ be the positive root of $\overline{\Phi}$ related to $s_i \in \overline{S}$ and set
$\overline{ B}: = \{\alpha_1, \ldots ,\alpha_n\}$.
Then set $$s_{1^*}:= (s_1, \alpha_1).$$

In the next lemma we explicitly write down the multiplication in $\mathcal{W}$.

\begin{Lemma}\label{lem:Multiplication}
The following hold.
\begin{itemize}
 \item[(a)]   
Let $w_i = (\overline{w_i},\lambda_i)  \in \mathcal{W}$ for $i=1,2$. Then 
 $w_1 w_2 = (\overline{w_1}\overline{w_2}, w_2^{-1}(\lambda_1) + \lambda_2) $.
 \item[(b)] $s_1 s_{1^*} = (1,\alpha_1) $ and  for $k \in \ZZ$ it is $(s_i, 0)(s_i, k\alpha_i) = (1, k \alpha_i)$.
 \end{itemize}
\end{Lemma}
\begin{proof}
  Assertion (a) follows from a direct calculation in $\mathcal{W}$ and (b) is a special case of (a).  
\end{proof}

Lemma~\ref{lem:Multiplication} gives us that $s_{1^*}$ is an 
involution: $s_{1^*}^2  = (s_1^2,s_1(\alpha_1) + \alpha_1) = (id, 0)$.

Set
$$\mathcal{S}:= \overline{S}\cup \{s_{1^*}\}\subset \mathcal{W},$$
and call it \defn{simple system for} $\mathcal{W}$, and $(\mathcal{W}, \mathcal{S})$ \defn{ extended Coxeter system}. 

Let $W$ and $W'$ be two groups that are generated by $S$ and $S'$, respectively.
We say that $(W,S)$ and $(W',S')$ are \defn{isomorphic}, if there is an isomorphism from
$W$ to $W'$ that maps $S$ bijectively onto $S'$.
As a consequence of Proposition~\ref{thm:isom_normalform} and \cite[Theorem~3.21]{BMW25} we obtain the next statement.

\begin{proposition}\label{Prop:CharExtendedWeylNotTubular}
The following hold.
\begin{itemize}
\item[(a)]
Let $(W,S)$ be an extended Weyl system of domestic or wild type, then it  is isomorphic to an extended  Coxeter system  of star type.
\item[(b)]
Let $(\mathcal{W}, \mathcal{S})$ be
an extended Coxeter system of star type where $\mathcal{W} = \mathcal{W}(\overline{W}, \overline{S})$. If $\overline{\Gamma}$ is not an affine diagram, 
then $(\mathcal{W}, \mathcal{S})$ is isomorphic to an extended Weyl system of domestic or wild type.
If $\overline{\Gamma}$ is an affine diagram, then $(\mathcal{W}, \mathcal{S})$ is isomorphic to the hyperbolic cover of an extended Weyl system of tubular type.
\end{itemize}
\end{proposition}
\begin{proof}
Statement (a) is a consequence of Proposition~\ref{thm:isom_normalform}
(a) and (b).

Let $\mathcal{W} = \mathcal{W}(\overline{W}, \overline{S})$ be as in (b). 
First assume that $\overline{\Gamma}$ is not of affine type. 
Then there is an extended Weyl group $(W,S)$ of  domestic or wild type such that 
$(\overline{W}, \overline{S})$  is isomorphic to the projected Coxeter system  of $(W,S)$ 
and such that  the map  $s_i$ to $s_i$, for $i \in \{1, \ldots ,n\}$, defines an isomorphism $\varphi$ from $(\overline{W}, \overline{S})$ to the  projected Coxeter system
(see Proposition and Definition~\ref{def:basic} (e)). It follows from 
Proposition~\ref{thm:isom_normalform} (a) that also $W$ is  isomorphic to a semi-direct
product of $\overline{W}$ with the root lattice $\overline{L} =L(\overline{\Phi})$, and by 
Proposition~\ref{thm:isom_normalform} (b) the action of $\overline{W}$ on 
$\overline{L}$ is the action of the Coxeter group $\overline{W}$ on its 
root lattice. Therefore, $\mathcal{W}$ 
and $W$ are isomorphic groups. In particular,  $\varphi$ can be extended to an isomorphism  between $(\mathcal{W}, \mathcal{S})$  and $(W,S)$ by sending $s_{1^*}$ to $s_{1^*}$, which is the first part of assertion (b).
The second part of (b) is Theorem~3.21 (b) of \cite{BMW25}.
 \end{proof}

We can conclude that 
the extended Coxeter systems  of star type and the hyperbolic covers of extended Weyl groups describe the same objects, see the next theorem (and
also \cite{ST97}).

\begin{theorem}\label{Thm:Isomorphy}
 Let $(\mathcal{W}, \mathcal{S})$ be an extended
Coxeter system  of star type.
Then there is an extended Weyl system $(W,S)$ such that its 
hyperbolic cover $(\tilde{W}, \tilde{S})$ is isomorphic to 
$(\mathcal{W}, \mathcal{S})$. Conversely, every hyperbolic cover of an extended Weyl system 
is isomorphic to an extended Coxeter system of star type.
\end{theorem}
\begin{proof}
 The first part of the assertion follows from  Proposition~\ref{Prop:CharExtendedWeylNotTubular} (b) and Proposition~\ref{Prop:hypcoverWild}. The second part is a consequence of 
 Proposition~\ref{Prop:CharExtendedWeylNotTubular} (a) and Proposition~\ref{Prop:hypcoverWild} if the extended Weyl system is of domestic or wild type. If the extended Weyl system is of tubular type, then it is \cite[Theorem~3.21 (a)]{BMW25}.
\end{proof}

\subsection{The group $\mathcal{W}$ as a reflection group}
The group $\mathcal{W} = \overline{W} \ltimes L(\overline{\Phi})$ acts as group of affine linear maps on $\overline{V}$, as noted by van der Lek \cite{Lek, Lek2}. There is also a linear representation of  $(\mathcal{W}, \mathcal{S})$, where the elements in $\mathcal{S}$ act as reflections.

%Let $(W,S)$ hte the Coxeter system related to the extended Coxeter system $(\mathcal{W}, \mathcal{S})$.
Let $(\tilde{W}, \tilde{S})$ be the hyperbolic cover of an extended Weyl group $(W,S)$, which is isomorphic to $(\mathcal{W}, \mathcal{S})$ (see Theorem~
\ref{Thm:Isomorphy}). Then $\tilde{W}$ acts faithfully on the $\RR$-space $\tilde{V}$,
and $\tilde{s}_i \in \tilde{S}$ acts as the reflection with respect to the root $\alpha_i \in B$ on 
$\tilde{V}$ (see Section~\ref{Sec:HyperbolicCover}). From this we conclude a linear representation for $\mathcal{W}$.

\begin{proposition}\label{Prop:LinearRepCoxEntended}
The map sending $s_i \in \mathcal{S}$  to $\tilde{s}_i \in \tilde{S}$ defines a faithful linear representation of $\mathcal{W}$ on $\tilde{V}$.
\end{proposition}
\begin{proof}
    This is an immeadiat consequence of Theorem~\ref{Thm:Isomorphy}.
\end{proof}

Recall that the extended Weyl system $(W,S)$ is defined as group of reflections on a space $V$, that 
 $\overline{V}$ is a subspace of $V$ and that $V = \overline{V} \oplus \RR a$.
(see Remark~\ref{Rem:RootSystem}). 

\medskip
\begin{Lemma}\label{Prop:LinearRepCoxEntended2}
The following hold.
\begin{itemize}
\item[(a)] The element $ws_iw^{-1} \in \mathcal{W}$ is the reflection of $\tilde{V}$ with respect to $w(\alpha_i)$.
\item[(b)] It is $\Phi = \mathcal{W}(\{\alpha_1, \ldots , \alpha_n, \alpha_{1^*}\}) = \overline{\Phi} + \ZZ a$.
\item[(c)] The set of reflections in $\mathcal{W}$ is 
$\mathcal{T}:= \bigcup_{w\in \mathcal{W}} w^{-1}\mathcal{S}w = 
\{(\overline{w} s_i \overline{w}^{-1}, k \overline{w}(\alpha_i)) \mid 1 \leq i \leq n~\mbox{and}~k \in \ZZ\}$.
\item[(d)] It is $s_i(a) = a$, for $i \in \{1, \ldots, n, 1^*\}$, and the projection 
$\tilde{V} \rightarrow \tilde{V}/ \RR a$  induces an epimorphism from $\mathcal{W}$
to $\overline{W}$.
\end{itemize}
\end{Lemma}
\begin{proof}
Assertion (a) holds by  Lemma~\ref{lem:Reflections} (b).

It is $\alpha_1 , \alpha_1 +a \in \Phi$. Application of $(1,k\alpha_1)$ to $\alpha_1$ and to $\alpha_1 +a$ yields that $\alpha_1 +\ZZ a \subset \Phi$.
Since $\overline{\Gamma}$ is a tree, all the roots in $\overline{\Phi}$ are
in one $\overline{W}$-orbit, and $\Phi$ is a subset of $\overline{\Phi} + \ZZ a$. As the latter set is invariant under the set of $\mathcal{W}$ generating reflections $\mathcal{S}$, equality holds, which 
shows (b).
Assertion (c) is a consequence of (a). Since $\Phi \subset \overline{V} + \RR a$, it is 
$s(a) = a$ for all $s \in \mathcal{S}$. Further notice, that $s_{1^*}$ equals $s_1$ in its action on $\mathcal{V}/\RR a$, which shows the last part of (d).
\end{proof}

\section{The Coxeter transformations}\label{sec:CoxTrans}

In this section we prove two properties for Coxeter transformations in extended Weyl systems (of all three types)
that also hold for Coxeter elements in (finite) Coxeter systems. First we show that, as in finite Coxeter systems, all the Coxeter transformations are conjugate in $W$. Then we will recall that for a Coxeter transformation $c \in W$ the length of $c$ with respect to the set $T$ of all reflections is $|S|$.

We  use the notation as introduced in Figure~\ref{def:GenCoxDiag} and abbreviate the simple reflection
with respect to the root $\alpha_{(i,j)}$ by $s_{ij}$. As defined in \ref{def:basic} (g), the element
$$
c = \Big(\prod_{i= 1}^r\prod_{j = 1}^{p_i} s_{ij}\Big) s_1s_{1^*} \in W
$$
is a Coxeter transformation of $(W,S)$, and every two Coxeter transformations only differ by the
ordering of the reflections $s_{ij}$,  where $1 \leq i \leq r$ and $1 \leq j \leq p_i$.

It is a consequence of  \cite[Chapter V, \S~6, Lemma~1]{Bou02} that in the definition of the Coxeter transformation the ordering of the reflections $s_{ij}$,  where $1 \leq i \leq r$ and $1 \leq j \leq p_i$, does not matter up to conjugacy. For the convenience of the reader we give a direct proof of this fact.

\begin{Lemma}\label{lem:ConjCoxTrans}
Let $(W,S)$ be an extended Weyl system. Then in $W$ all the Coxeter transformations  are conjugated.
\end{Lemma}
\begin{proof}
Let $d$ be another Coxeter transformation of $(W,S)$. Because of the shape of the diagram and the definition of a Coxter transformation, we have
$$
d = \Big(\prod_{i= 1}^r\prod_{j = 1}^{p_i} s_{i\pi_i(j)} \Big) s_1s_{1^*},
$$
where $\pi_i$ is in $\Sym(p_i)$ for $1 \leq i \leq r$. Thus $c_i:= \prod_{j = 1}^{p_i} s_{ij}$ and $d_i:= \prod_{j = 1}^{p_i} s_{i \pi_i(j)}$
are Coxeter elements in a Coxeter system $\Lambda_i$ of type $A_{p_i}$ with set of simple reflections
$S_i:= \{s_{ij}~|~1 \leq j \leq p_i\}$ for every $1 \leq i \leq r$.

We claim that $P_i:= \langle s_{i2}, \ldots , s_{ip_i}\rangle$ acts transitively on the set of Coxeter elements in
$W_i: = \langle S_i\rangle$, more precisely that 
there is $x_i \in P_i$ such that $c_i^{x_i} = d_i$. 
As $(W_i, S_i)$ is a Coxeter system of type $A_{p_i}$,
the Coxeter elements $c_i$ and $d_i$ are $(p_i+1)$-cycles in 
$W_i \cong \Sym(p_i+1)$. 
Without loss of generality we can assume that $c_i = (1~2~\ldots ~p_i+1)$ and $d_i = ( 1 ~ k_2 ~ \ldots ~ k_{p_i+1})$. Then $x_i \in P_i$ defined by $x_i(j) := k_j$ for $2 \leq j \leq p_i+1$ maps $c_i$ onto $d_i$.

Now it follows, as every element of $P_i$ commutes with $s_1$, $s_{1^*}$ as well as with every element in $S_k$ for $k \neq i$, that $x: = x_1\cdots x_r$ conjugates $c$ onto $d$.
\end{proof}

We recall the reflection length of a Coxeter transformation $c$ of the extended Weyl group $W$, that is, we recall the number
$$
\ell_{T}(c):=\min\lbrace n\in \NN_{0}\mid t_{1}\cdots t_{n}=c,~t_{i}\in T \text{ for }1\leq i \leq n \rbrace.
$$
It  will be an ingredient of the proof of Theorem~\ref{thm:main}.

\begin{proposition} \label{cor:RefLength} 
Let $W$ be an extended Weyl system  of rank $n$. Then $\ell_T(c) = n$ for every Coxeter transformation $c \in W$. 
\end{proposition}
\begin{proof}
We follow the argument in \cite{BMW25}. We can apply Scherk's theorem to 
$\tilde{c} \in \tilde{W}$ \cite{ST89}, and obtain $n = \ell_{\tilde{T}}(\tilde{c}) = \ell_T(c)  $ by Proposition~\ref{Prop:hypcoverWild} (and in the elliptic case by \cite{BMW25}).
\end{proof}
\section{A remark on Coxeter systems whose diagram is a star} \label{sec:AuxProofMain}

In this section we present some information on the root system $\overline{\Phi}$ for a Coxeter system $(\overline{W}, \overline{S})$ whose diagram is a star.
We use the notation as introduced in Definition~\ref{def:basic} (c). In particular, $\overline{B}$  is  a simple system for $\overline{\Phi}$, and we use the numbering of the roots in  $\overline{B}$ as given in Figure~1. The reflection  with respect to the root $\alpha_{(i,j)} \in \overline{B}$ will again be abbreviated by $s_{ij}$. The results obtained in this section will be used in the proof of Theorem~\ref{thm:main}.
We start with a little observation.

\begin{Lemma}\label{lem:GenCoxGrp}
Let $\beta_1, \ldots, \beta_n \in \overline{\Phi}$ such that $\overline{W }= \langle s_{\beta_1},\ldots,s_{\beta_{n}} \rangle $. Then $L(\overline{\Phi})= L(\{\beta_{1},\ldots,\beta_{n}\})$.
\end{Lemma}
\begin{proof}
As the Coxeter diagram is a tree, all the  reflections are conjugated in $\overline{W}$.  In particular, for every $\beta\in \overline{\Phi}$ there exists $w\in \overline{W}=\langle s_{\beta_{1}},\ldots, s_{\beta_{n}} \rangle$ such that $s_{w(\beta_{1})}=ws_{\beta_{1}}w^{-1}=s_{\beta}$. It follows that $\beta=\pm w(\beta_{1}) \in \spanz(\beta_{1},\ldots,\beta_{n})$.
Therefore  $L(\overline{\Phi})\subseteq \spanz(\beta_{1},\ldots,\beta_{n})$.
\end{proof}

\begin{Lemma}\label{lem:RootsHyperbolic}
Let $\alpha \in \overline{\Phi}$ such that $\displaystyle \alpha = \alpha_1 + \sum_{i= 1}^{r}\sum_{j = 1}^{p_i} \lambda_{ij} \alpha_{(i,j)}$,
where $\lambda_{ij} \in \NN_0$. Then for each $i \in \{1, \ldots r\}$ either
\begin{itemize}
\item[(a)]  $\lambda_{ij} = 0$ for $1 \leq j \leq p_i$, or
\item[(b)]  there is $m_i \in \{1, \ldots , p_i \}$ such that $\lambda_{ij} = 1$ for $1 \leq j \leq m_i$ and $\lambda_{ij} = 0$ for $m_i < j \leq p_i$.
\end{itemize}
\end{Lemma}
\begin{proof}
First suppose that $\lambda_{\ell 1} = 0$ for some $\ell \in  \{1, \ldots ,r\}$. Furthermore, assume that $\lambda_{\ell m} \neq 0$ for some  $m \in \{2, \ldots , p_{\ell} \}$. Without loss of generality we may assume that $m$ is maximal with this property, that is, $\lambda_{\ell j} = 0$ for all $j >m$. We calculate
$$
s_{\ell 2} \cdots s_{\ell m}(\alpha) =  \alpha_1 - \lambda_{\ell m}\alpha_{(\ell,1)} +\sum_{j = 2}^{m}( \lambda_{\ell, j-1}  - \lambda_{\ell m}) \alpha_{(\ell,j)}
+ \sum_{i\neq \ell}^r \sum_{j = 1}^{p_i} \lambda_{ij} \alpha_{(i,j)}.
$$
Since $s_{\ell 2} \cdots s_{\ell m}(\alpha)$ is a positive root it follows $ \lambda_{\ell m} = 0$, contrary to our assumption. This shows (a).

Now assume $\lambda_{\ell 1} \neq 0$. Further  let $m \in \{2, \ldots , p_{\ell}\} $ such that $\lambda_{\ell m} \neq 0$, but $\lambda_{\ell j} = 0$ for all $j >m$. Then 
$$s_{\ell 1}\cdots s_{\ell m} (\alpha) = \alpha_1 + (1- \lambda_{\ell m}) \alpha_{(\ell,1)} + \sum_{j = 2}^m(\lambda_{\ell, j-1} - \lambda_{\ell m}) \alpha_{(\ell,j)} + \sum_{ i \neq \ell}^r \sum_{j = 1}^{p_i} \lambda_{ij} \alpha_{(i,j)}.$$ 
Since $s_{\ell 2} \cdots s_{\ell m}(\alpha)$ is a positive root and $\lambda_{\ell m} \geq 1$ we obtain $\lambda_{\ell m} = 1$, hence $(1- \lambda_{\ell m}) = 0$. As in the first part of this proof, this yields 
$\lambda_{\ell, j-1} - \lambda_{\ell m} = 0$ for $j \in \{ 2, \ldots , m \}$, which shows (b).
\end{proof}

\begin{corollary}\label{Coro:conj}
The parabolic subgroup  $P: = \langle s_{ij}~|~1 \leq i \leq r, 1 \leq j \leq p_i\rangle$ of $\overline{W}$ acts transitively on the set of roots $\left(\alpha_1 +  \Lambda \right) \cap \overline{\Phi}$, where $\Lambda:= \spanz ( \alpha_{(1,1)},\ldots,\alpha_{(r,p_r)})$. In particular, $P$ acts by conjugation transitively on the set of reflections $\{t \in T \mid W = \langle P, t\rangle \}$.
 \end{corollary}
 
 \begin{proof}
Let  $\beta = \sum_{i= 1}^{r}\sum_{j = 1}^{p_i} \lambda_{ij} \alpha_{(i,j)} \in \Lambda$ and $\alpha = \alpha_1 + \beta$.
If $\lambda_{i1} \neq 0$ for some $1 \leq i \leq r$, then there is $m_i \in \{1, \ldots , p_i \}$ such that $\lambda_{ij} = 1$ for $1 \leq j \leq m_i$ and $\lambda_{ij} = 0$ for $m_i \leq j \leq p_{i}$
by Lemma~\ref{lem:RootsHyperbolic}.  If we set $w: = s_{\ell 1}\cdots s_{\ell m}$ then we get, as in the proof of  Lemma~\ref{lem:RootsHyperbolic}, that $w(\alpha)$ is a linear combination
of simple roots that do not contain a root $\alpha_{(i,j)}$ for some $1 \leq j \leq p_{i}$.  Now the assertion follows by induction on $r$.

We obtain the last assertion as the set of reflections with respect to the set of roots $\left(\alpha_1 +  \Lambda \right) \cap \overline{\Phi}$ equals $\{t \in T \mid W = \langle P, t\rangle \}$ by Lemma~\ref{lem:GenCoxGrp}.
\end{proof} 

\section{Hurwitz transitivity for Coxeter transformations in extended Coxeter systems of star type} \label{sec:ProofMain}

The aim of this section is to prove the main results of this paper, that is 
Theorems~\ref{thm:main} and \ref{thm:AllgAussagehyperCover}.
First we collect some results on the Hurwitz action in 
Coxeter groups that we will need in the proof.

\subsection{Hurwitz action in Coxeter systems} \label{subsec:HurwitzInCoxeter}

Let $(W,S)$ be an arbitrary Coxeter system of finite rank.
Let as usual be the length of $w \in W$ with respect to $S$
$$\ell_S(w):=\min \lbrace k\in \NN_{0}\mid w = s_1 \cdots s_k, ~s_1, \ldots , s_k \in S\rbrace.$$
Further let $T:=\bigcup_{w\in W}wSw^{-1}$
be the set of reflections in $(W,S)$, and  $\ell_T$ the length function on $W$ with respect to $T$. A {\em parabolic Coxeter element} is an element in $W$
that is a Coxeter element in a parabolic subgroup of $W$ (see \cite{BDSW14}).
Let $c_0$ be a parabolic Coxeter element of $W$ of length $\ell_T(c_0) = n$.
Then the set of reduced reflection factorizations of $c_0$ in $W$ is
$$
\text{Red}_{T}(c_0):=\left\lbrace (t_{1},\ldots,t_{n})\in T^{n}\mid c_0=t_{1}\cdots t_{n}\right\rbrace.
$$
Next  we quote three results on the Hurwitz action in Coxeter systems, whose proofs can be
found at the given references.

\begin{theorem}[{\cite[Theorem~1.3]{BDSW14}}]\label{thm:HurwitzInCox}
Let $(W,S)$ be a Coxeter system and $c_0$ a parabolic Coxeter element in $W$.
Then the Hurwitz action of ${\mathcal B}_n$ is  transitive on $\Red_{T}(c_0)$.
\end{theorem}

The following lemma is a result about the Hurwitz orbits of non-reduced reflection factorizations. 

\begin{Lemma}[{\cite[Lemma 2.3]{WY19}}] \label{lem:same_refl}
Let $(W,S)$ be a Coxeter system, $w \in W$ and   $w=t_{1}\cdots t_{n+2k}$ where $n = \ell_S(w)$,  $k \in \NN$ and  $t_{i}\in T$ for $1\leq i \leq n+2k$. 
Then there exists a braid $\sigma \in \mathcal{B}_{n+2k}$ 
and reflections $r_{1},\ldots,r_{n},r_{i_1},\ldots ,r_{i_k} \in T$
such that 
$$
\sigma(t_{1},\ldots,t_{n+2k})=(r_{1},\ldots,r_{n},r_{i_1},r_{i_1}, \ldots , r_{i_k},r_{i_k}).
$$
\end{Lemma}

The proof of the next lemma also appears  implicitly  in the proof of \cite[Theorem 1.1]{LR16}.

\begin{Lemma}[{\cite[Lemma 2.6]{WY19}}] \label{lem:conj_hurw}
Let $(W,S)$ be a Coxeter system and let $t_1, \ldots, t_{n}, t \in T$. Then $(t_1, \ldots, t_{n},t,t)$ and $(t_1, \ldots, t_{n}, t^w,t^w)$ lie in the same Hurwitz orbit for every $w \in \langle t_1, \ldots, t_{n} \rangle$. 
\end{Lemma}

\medskip

 \subsection{Hurwitz action in extended Coxeter systems}\label{Sec:HAinCoxExtendedRootLattice}

In this section we prove transitive Hurwitz action in extended Coxeter systems of star type (Theorem~\ref{thm:AllgAussage}), which establishes a uniform approach to transitive Hurwitz action in the hyperbolic covers of the extended Weyl groups.
We  use the notation introduced in Section~\ref{Sec:CoxExtensionRootLattice}.
Let $(\overline{W}, \overline{S})$ be a Coxeter system, whose Coxeter diagram $\overline{\Gamma}$ is a star, let
$\overline{S} = \{s_1, \ldots, s_n\}$ be such that $1$ is the label of a vertex of maximal degree in $\overline{\Gamma}$, and let $\mathcal{S} = \overline{S}
\cup \{s_{1*}\}$. Then $(\mathcal{W}, \mathcal{S})$ is 
an extended Coxeter system of  star type.

\begin{remark}\label{rem:Choice1}
There is more than one vertex $v$  of maximal degree in $\overline{\Gamma}$ precisely if $\overline{\Gamma}$ is of type 
$A_n$. In this case each vertex of  the diagram is a possible choice for  the label $1$.
\end{remark}

Analogously to Section~\ref{sec:CoxTrans} we define a 
Coxeter transformation $c$ in $\mathcal{W}$ by
$$c:= c_0s_1s_{1^*} ~\mbox{where}~ c_0 = \prod_{s \in \overline{S}\setminus{\{s_1\}}} s$$ is a Coxeter element in the parabolic subgroup $P:= \langle s_2, \ldots , s_n\rangle$ of $\overline{W}$.
%, which is itself a Coxeter group.

\begin{Lemma} \label{cor:RefLengthExtCox} 
Let $(\mathcal{W}, \mathcal{S})$ be an extended Coxeter system. Then $\ell_{\mathcal{T}}(c) = |\mathcal{S}| = n+1$. 
\end{Lemma}
\begin{proof}
This is a consequence of Proposition~\ref{Prop:CharExtendedWeylNotTubular}
and Proposition~\ref{cor:RefLength} for the domestic and wild cases, and
\cite[Proposition~1.2]{BMW25} for the tubular case.
\end{proof}

 We set 
$$ \Red_{\mathcal{T}}(c):=\{(t_1, \ldots , t_{n+1}) \in \mathcal{T}^{n+1} \mid c =t_1 \cdots t_{n+1}\}.$$

In this section we will prove the following theorem, which is stronger than Theorem~\ref{thm:main}.

\begin{theorem}\label{thm:AllgAussage}
Let $(\mathcal{W}, \mathcal{S})$ be an extended  Coxeter 
system  of star type with set of reflections $\mathcal{T}$ and Coxeter transformation $c$. Then the Hurwitz action is transitive on $\Red_{\mathcal{T}}(c)$.
    \end{theorem}

Throughout this section we use several times the projection map from $\mathcal{W}$ to
$\overline{W}$ sending $w = (\overline{w},\TR(w))$ onto $\overline{w}$. The image $\overline{W}$ acts faithfully on $\overline{V} \subset \mathcal{V}$ and the reflection in $\mathcal{W}$ with respect 
to the root $\beta  = \overline{\beta} +k a\in \Phi$ where $\overline{\beta}$ is in $\overline{\Phi}$ and $k\in \ZZ$ maps to the reflection of $\overline{V}$ with respect to the root 
$\overline{\beta}$. In the following for a root $\beta \in \Phi$ we will write $s_\beta$ for the reflection of $\mathcal{V}$ with respect to $\beta$.
Before we prove Theorem~\ref{thm:AllgAussage}  we consider a special situation.

\begin{Lemma}\label{lem:trans_rk2}
Let $s_{\beta_{1}}s_{\beta_{2}}=s_{1}s_{1^{*}}$ with $\beta_{1},\beta_{2}\in \Phi$. Then $(s_{\beta_{1}},s_{\beta_{2}})$ and $(s_{1},s_{1^{*}})$ lie in the same Hurwitz orbit, that is, $\mathcal{B}_{2}(s_{1},s_{1^{*}})=\lbrace (t_{1},t_{2})\in \mathcal{T}^{2} \mid t_{1}t_{2}=s_{1}s_{1^{*}}\rbrace$.
\end{Lemma}
\begin{proof}
By considering the projection we get $\overline{s_{\beta_{1}}}\overline{s_{\beta_{2}}}=\overline{s_{1}}\overline{s_{1^{*}}}=1\in \overline{W}$. Therefore, we get $\overline{\beta_{1}}=\pm \overline{\beta_{2}}$
and we can assume
 $\beta:=\overline{\beta_{1}}=\overline{\beta_{2}}\in \overline{\Phi}^{+}=\text{span}_{\ZZ_{\geq 0}}(\overline{B})\cap \overline{\Phi}$. We have $\TR(s_{\beta_{i}})=\lambda_{i}\beta,~\lambda_{i}\in \ZZ$ for $i=1,2$ (see Lemma~\ref{lem:Multiplication}). Therefore we obtain by applying 
Lemma~\ref{lem:Multiplication} (a)
$$
\alpha_{1} = \TR (s_{1}s_{1^{*}})=\TR(s_{\beta_{1}}s_{\beta_{2}}) =
 s_{\beta}\left(\TR(s_{\beta_{1}})\right)+\TR(s_{\beta_{2}}) = 
s_{\beta}(\lambda_{1}\beta)+\lambda_{2}\beta =
 (\lambda_{2}-\lambda_{1})\beta.
$$
As the root system $\overline{\Phi}$ is reduced we get $\beta=\pm \alpha_{1}$. 
By Lemma~\ref{lem:Multiplication} (b)
$s_{\beta_{1}}$ and $s_{\beta_{2}}$ are reflections of the infinite dihedral group $\widetilde{W}:= \left\langle s_{1},s_{1^{*}} \right\rangle$ with simple system $S':=\lbrace s_{1},s_{1^{*}}\rbrace$. Therefore $s_{\beta_{1}}s_{\beta_{2}}$ is a reduced reflection factorization of the Coxeter element $s_{1}s_{1^{*}}$ in $\widetilde{W}$. By \cite[Theorem 1.3]{BDSW14} the factorizations $(s_{\beta_{1}},s_{\beta_{2}})$ and $(s_{1},s_{1^{*}})$ lie in the same Hurwitz orbit.
\end{proof}

\medskip
\noindent
{\bf Proof of Theorem~\ref{thm:AllgAussage}}
We number $\overline{S}$ such that the Coxeter element becomes $c=s_{2}\cdots s_{n-1}s_{1}s_{1^{*}}$.
(Notice also that by the argument in the proof of Lemma~\ref{lem:ConjCoxTrans} 
a Coxeter element with a different ordering of $s_{2},\ldots ,s_{n-1}$ is conjugate to $c$.)
Let $\alpha_{i}\in \overline{\Phi}$ such that $s_{\alpha_{i}}=s_{i}$ for $1\leq i \leq n$, and $s_{\alpha_{1^{*}}}=s_{1^{*}}$ as introduced in Section~\ref{Sec:HAinCoxExtendedRootLattice}. Fix a reduced factorization $(t_{1},\ldots,t_{n+1})\in R_{\mathcal{T}}(c)$. 
We prove the theorem by showing that there exists a braid $\tau\in \mathcal{B}_{n+1}$ such that 
$$
\tau (t_1, \ldots , t_{n+1})=(s_{1},\ldots,s_{n-1},s_{1},s_{1^{*}}).
$$

Consider the factorization $(\overline{t_{1}},\ldots,\overline{t_{n+1}})\in \overline{W}^{n+1}$ where $\overline{t_{i}}\in \overline{T}$ for $1\leq i \leq n+1$ that is induced by $(t_1, \ldots , t_{n+1}) \in R_{\mathcal{T}}(c)$. In $\overline{W}$ we
calculate
\[\overline{t_{1}}\cdots \overline{t_{n+1}}=\overline{c}=s_{2}\cdots s_{n-1}s_{1}\overline{s_{1^{*}}}=s_{2}\cdots s_{n-1},\]
where $\overline{s_{i}}=s_{i}$ for $1\leq i \leq n$ and $\overline{s_{1^{*}}}=s_{1}$. Since $\ell_{\overline{S}}(\overline{c})=n-1$, Lemma \ref{lem:same_refl} yields the existence of a braid $\tau_{1}\in \mathcal{B}_{n+1}$ and a reflection $t \in \overline{T}$ such that 
$$
\tau_{1}(\overline{t_{1}},\ldots,\overline{t_{n+1}})=(\overline{t'_{1}},\ldots,\overline{t'_{n-1}},t,t).
$$ 
Thus we have 
\[\overline{t'_{1}} \cdots \overline{t'_{n-1}}=\overline{t'_{1}} \cdots \overline{t'_{n-1}}tt=\overline{c}=s_{2}\cdots s_{n}.\]

Therefore, $\overline{t'_{1}} \cdots \overline{t'_{n-1}}$ is a reduced reflection factorization of the parabolic Coxeter element $\overline{c}$ of the Coxeter system $(\overline{W},\overline{S})$. By Theorem~\ref{thm:HurwitzInCox} there exists a braid $\tau_{2}\in \mathcal{B}_{n+1}$ such that 
$$
\tau_{2}(\overline{t'_{1}},\ldots,\overline{t'_{n-1}},t,t)=(s_{2},\ldots,s_{n},t,t).
$$ 
Applying the braid $\tau_{2}\tau_{1}$ to the initial factorization and using Lemma~\ref{Prop:LinearRepCoxEntended2}
(d), we obtain
\[\tau_{2}\tau_{1}(t_{1},\ldots,t_{n-1})=(t''_{1},\ldots,t''_{n-1},t_{a},t_{b})\]
where $t''_{1},\ldots,t''_{n-1},t_{a},t_{b} \in \mathcal{T}$, and
$\overline{t''_{i}}=s_{i+1}$ for $1\leq i \leq n-1$, and $\overline{t_{a}}=\overline{t_{b}} = t$.
Let $\alpha \in  \overline{\Phi}^{+}$ be such that    
$t=s_{\alpha}$. Then the normal form of $t_at_b$ is 
$(1, \lambda \alpha)$
for some $\lambda \in \ZZ$ by Lemma~\ref{lem:Multiplication} (a).

Therefore we obtain using Lemma~\ref{lem:Multiplication} (a), and setting
$\beta:=  \TR (t''_{1}\cdots t''_{n-1})$:
$$
\alpha_{1} = \TR(c) = \TR (t''_{1}\cdots t''_{n-1}t_{a}t_{b}) =
\beta + \TR (t_{a}t_{b}) =
\beta+\lambda \alpha.
 $$
The fact that $\overline{t''_{i}}=s_{i+1}$ for $1\leq i \leq n-1$ yields 
$\overline{t''_{1}}\cdots \overline{t''_{n-1}} \in P:= \langle  s_{2}, \ldots,
s_{n}\rangle \leq \overline{W}$ and 
$\beta\in \spanz(\alpha_{2},\ldots,\alpha_{n-1})$, say $\beta= \sum_{i=2}^{n} \lambda_i \alpha_i$ with $\lambda_i \in \ZZ$ for $2 \leq i \leq n$. As $\lbrace \alpha_{1},\alpha_{2},\ldots,\alpha_{n}\rbrace$ is linear independent in the $\RR$-span of $L(\overline{\Phi})$, it follows $\lambda\neq 0$ and therefore
\[\alpha =\frac{1}{\lambda}\alpha_{1}-\frac{1}{\lambda} \beta 
= \frac{1}{\lambda}\alpha_{1} - \sum_{i=2}^{n} \frac{\lambda_i }{\lambda}\alpha_i \in \overline{\Phi}^{+}.\]
The latter implies, as $\alpha \in \overline{\Phi}^+$ and $\overline{B} = \{\alpha_1, \ldots \alpha_n\}$ is a set of simple roots for the Coxeter group $\overline{W}$, that  $\lambda=1$ (see \cite[Proposition~2.1]{Deod82}). Hence $\alpha =\alpha_{1}-\beta$. 

By Corollary~\ref{Coro:conj} there exists $x\in P$ such that $t^{x}=s_{\alpha}^{x}=s_{1}$ and therefore we get  $\overline{t_{a}^{x}}=\overline{t_{a}}^{x}=s_{1}$ with Lemma~\ref{Prop:LinearRepCoxEntended2} (d). Since $\overline{t''_{i}}=s_{i+1}$ for $1\leq i\leq n-1$, Lemma \ref{lem:conj_hurw} yields the existence of a braid $\tau_{3}\in \mathcal{B}_{m}$ such that
$$
\tau_{3}(t''_{1},\ldots,t''_{m-1},t_{a},t_{b})=(t''_{1},\ldots,t''_{m-1},t^x_{a},t^x_{b}).
$$
Set $t'_a = t^x_a$ and $t'_b = t^x_b$, and observe that $\TR(t'_{a} t'_{b})=\lambda' \alpha_{1}$ for some $\lambda' \in \ZZ$.

Similiar as above we can use Lemma~\ref{lem:Multiplication} (a)
 to obtain
$$
\alpha_{1} =  \TR (c) = \TR(t''_{1}\cdots t''_{m-2} t'_{a} t'_{b}) = \beta +\TR (t'_{a} t'_{b}) = 
 \beta +\lambda' \alpha_{1}.
$$
This yields, as $\alpha_1$ and $\beta$ are linear independent, that  $\beta =0$ and $\lambda' = 1$. 

Since the roots related to the reflections $\overline{t''_{i}}$, where $1\leq i \leq n-1$,
are linearly independent, Lemma \ref{lem:zero_tran} yields that
\[\TR(t''_{1})=\ldots=\TR(t''_{m-2})=0,\]
and therefore $t''_{i}=s_{i+1}$ for $1\leq i \leq n-1$. From  $\overline{t'_{a}} = \overline{t'_{a}} = s_1$ and $\lambda'= 1$ follows $t'_{a}t'_{b}=s_{1}s_{1^{*}}$.  Therefore  by Lemma \ref{lem:trans_rk2} there exists a braid $\tau_{4}\in \mathcal{B}_{n+1}$ such that
\[\tau_{4}(s_{2},\ldots,s_{n-1},t'_{a},t'_{b})=(s_{2},\ldots,s_{n-1},s_{1},s_{1^{*}}).\]
Altogether, setting $\tau:=\tau_{4}\tau_{3}\tau_{2}\tau_{1}\in \mathcal{B}_{m}$, we obtain
\begin{align*}
\tau(t_{1},\ldots,t_{m})=(s_{1},\ldots,s_{n-1},s_{1},s_{1^{*}}),
\end{align*}
which proves the assertion.
 \hfill $\square$
 
 \medskip
\noindent
{\bf Proof of Theorem~\ref{thm:main}}
Every extended Weyl system $(W,S)$ of domestic or wild type is isomorphic to an extended Coxeter system of star type, and  
the set of reflections of $W$ is mapped onto  the set of 
reflections in the extended Coxeter group  (see Proposition~\ref{Prop:CharExtendedWeylNotTubular} (a)). Further every Coxeter transformation $c \in W$ is mapped onto a  Coxeter transformation in the
extended Coxeter group, and  
$\Red_T(c)$ is mapped onto $\Red_{\mathcal{T}}(c)$.
Therefore the assertion is a consequence of  Theorem~\ref{thm:AllgAussage}.
\hfill $\square$
\section{Two applications of the main theorem}\label{sec:application}

\subsection{The category of coherent sheaves over a weighted projective line}\label{subsec:apprep}
In the representation theory of finite dimensional algebras hereditary categories are an important tool, since they serve as prototypes for many phenomena appearing there. An important class of these objects are the hereditary ext-finite abelian $k-$categories with tilting object where $k$ is an algebraically closed field. They are classified up to derived equivalence by Happel in \cite{Hap01}. Up to derived equivalence they are the categories of coherent sheaves over a weighted projective line in the sense of Geigle and Lenzing \cite{GL87} and the categories of finite dimensional modules over a hereditary finite dimensional $k-$algebra.

One way to study their structure is to investigate the lattice of thick subcategories.

\begin{definition}\label{def:2outof3}
  Let $\mathcal{A}$ be an abelian category and $\mathcal{H}$ a full subcategory of $\mathcal{A}$. The category $\mathcal{H}$ is called \defn{thick} if it is abelian and closed under extensions.

If $H$ is a collection of objects in $\mathcal{A}$, we denote the smallest thick subcategory that contains $H$ by $\Thi(H)$ and call it the \defn{thick subcategory generated by $H$}.
\end{definition}

In \cite{BWY21} the authors show that one can attach to the category of coherent sheaves over a weighted projective line an extended Weyl group by proving the existence of a complete exceptional sequence that induces an extended Coxeter-Dynkin diagram as given in Figure \ref{def:GenCoxDiag}.
This information can be used to 'classify' to some extend all the thick subcategories of an hereditary ext-finite abelian category with tilting object that are generated by an exceptional sequence, as we will desribe next. First we
include the neccessary definitions.

\begin{definition}
In an abelian category $\mathcal{A}$ an object $E$  is called \defn{exceptional} if $\END_{\mathcal{A}}(E)=k$ and $\EXT_{\mathcal{A}}^{i}(E,E)=0$ for $i>0$. A pair $(E,F)$ of exceptional objects in $\mathcal{A}$ is called \defn{exceptional pair} provided it holds in addition $\EXT_{\mathcal{A}}^{i}(F,E)=0$ for $i\geq 0$. A sequence $\mathcal{E}=(E_{1},\ldots,E_{r})$ of exceptional objects in $\mathcal{A}$ is called \defn{exceptional sequence of length $r$} if $\EXT_{\mathcal{A}}^{s}(E_{j},E_{i})=0$ for $i<j$ and all $s>0$. An exceptional sequence in $\mathcal{A}$ is called \defn{complete} if the smallest thick subcategory which contain the sequence is $\mathcal{A}$.
\end{definition}

In order to formulate our 'classification' of the thick subcategories that are generated by an exceptional sequence we need some 
further notation in the extended Weyl group.

\begin{definition}\label{Def:Generating}
Let $(W,S)$ be an extended Weyl system, $T$ its set of reflections and $c$ a Coxeter transformation. 
Define a partial order on $W$, the so called \defn{absolute order}, by
\[x\leq y \text{ if  } \ell_{T}(x)=\ell_{T}(x)+\ell_{T}(x^{-1}y).\]

We let $[\idop,c]=\lbrace x\in W\mid \idop\leq x \leq c \rbrace$, and call the poset $([\idop,c], \leq )$ \defn{interval poset}.
The poset $[\idop,c]$ consists of the prefixes of the reduced reflection factorizations of $c$. We call the factorization $(t_{1},\ldots, t_{n})$ of $c$ \defn{generating}, if $W=\langle t_{1},\ldots,t_{n} \rangle$. The poset
$[\idop, c]^{\mathrm{gen}}$ is the subposet of $[\idop, c]$ whose elements are the prefixes of the reduced generating factorizations of $c$ ordered by the absolute order.
\end{definition}

An easy consequence of our main result Theorem \ref{thm:main} is that $[\idop,c]= [\idop,c]^{\mathrm{gen}}$ for a Coxeter transformation $c\in W$ and $W$ of domestic or wild type. For $W$ of tubular type $E_6^{(1,1)}$ the authors show in \cite{BWY21} that $[\idop,c]\neq [\idop,c]^{\mathrm{gen}}$ for each Coxeter transformation.
But notice that for $W$ of all three types, in $\tilde{W}$ it is always $[\idop,c]= [\idop,c]^{\mathrm{gen}}$ for a Coxeter transformation $c \in \tilde{W}$  by Theorem~\ref{thm:AllgAussagehyperCover}.
 
 \medskip
 
\noindent
{\bf Proof of Theorem~\ref{thm:mainRep}}
Let $\XX$ be a weighted projective line over an algebraically closed field $k$ of characteristic zero, $\COH(\XX)$ the category of coherent sheaves over $\XX$, $\Phi$ the associated generalized root system, $\tilde{W}$ the hyperbolic  cover of the associated extended Weyl group $W$ and $c \in \tilde{W}$ a Coxeter transformation.
As in  $\tilde{W}$ every reduced factorization of   $c$ is generating by Theorem~\ref{thm:AllgAussagehyperCover},
Condition (b) of \cite[Theorem 1.2]{BWY21} is satisfied for $(\tilde{W}, c)$. Condition (a) follows directly from Theorem~\ref{thm:AllgAussagehyperCover}.
Thus we conclude Theorem~\ref{thm:mainRep} from \cite[Theorem 1.2]{BWY21}. \qed

\subsection{Singularity theory} \label{subsec:singularity}
For details and definitions we refer to \cite{Ebe19} and the references therein.

Let $f: (\mathbb{C}^{k+1}, 0) \rightarrow (\mathbb{C}, 0)$ be a holomorphic function germ with an isolated singularity at the origin and let $k$ be even. Further assume that $f$ defines a \defn{simple elliptic singularity} in the sense of Saito \cite{Sai74} or a \defn{hyperbolic singularity}, that is, a singularity (up to stable equivalence) of the series 
$$
T_{p,q,r}: ~x^p+y^q+z^r+axyz,
$$
where $a \neq 0$, $2 \leq p \leq q \leq r$ and $\frac{1}{p}+\frac{1}{q}+\frac{1}{r}<1$ (see also \cite{Arn73}). These singularities are unimodal and by \cite{Gab74} a Coxeter-Dynkin diagram $\Gamma$ with respect to a distuingished base is given by a diagram as in Figure \ref{def:TreeSing}.
\begin{figure}[h!]
  \centering
  \begin{tikzpicture}[scale=3]

    \node (A6666) at (0,-4) [circle, draw, fill=black!50, inner sep=0pt, minimum width=4pt] {};
    \node (AA6666) at (0,-3.9) [] {\tiny{$(1,p-1)$}};
    \node (A666) at (0.5,-4) [circle, draw, fill=black!50, inner sep=0pt, minimum width=4pt] {};
    \node (AA666) at (0.5,-3.9) [] {\tiny{$(1,p-2)$}};
    \node (A66) at (0.75,-4) [] {$\ldots$};
    \node (A6) at (1,-4) [circle, draw, fill=black!50, inner sep=0pt, minimum width=4pt] {};
    \node (AA6) at (1,-3.9) [] {\tiny{$(1,2)$}};
    \node (B6) at (1.5,-4) [circle, draw, fill=black!50, inner sep=0pt, minimum width=4pt]{};
    \node (BB6) at (1.45,-3.9) [] {\tiny{$(1,1)$}};
    \node (C6) at (2,-4) [circle, draw, fill=black!50, inner sep=0pt, minimum width=4pt]{};
    \node (CC6) at (2,-4.1) [] {\tiny{$1$}};
    \node (D6) at (2.5,-4) [circle, draw, fill=black!50, inner sep=0pt, minimum width=4pt]{};
    \node (DD6) at (2.55,-3.9) [] {\tiny{$(3,1)$}};
    \node (E6) at (3,-4) [circle, draw, fill=black!50, inner sep=0pt, minimum width=4pt]{};
    \node (EE6) at (3,-3.9) [] {\tiny{$(3,2)$}};
    \node (E66) at (3.25,-4) [] {$\ldots$};
    \node (F6) at (3.5,-4) [circle, draw, fill=black!50, inner sep=0pt, minimum width=4pt]{};
    \node (FF6) at (3.5,-3.9) [] {\tiny{$(3,r-2)$}};
    \node (G6) at (4,-4) [circle, draw, fill=black!50, inner sep=0pt, minimum width=4pt]{};
    \node (GG6) at (4,-3.9) [] {\tiny{$(3,r-1)$}};
    \node (I6) at (2,-3.5) [circle, draw, fill=black!50, inner sep=0pt, minimum width=4pt]{};
    \node (II6) at (2,-3.4) [] {\tiny{$1^*$}};
    \node (K6) at (1.65,-4.2) [circle, draw, fill=black!50, inner sep=0pt, minimum width=4pt]{};
    \node (KK6) at (1.55,-4.15) []{\tiny{$(2,1)$}};
    \node (K66) at (1.3,-4.4) [circle, draw, fill=black!50, inner sep=0pt, minimum width=4pt]{};
    \node (KK66) at (1.2,-4.35) []{\tiny{$(2,2)$}};
    \node (K66666) at (1.12,-4.5) []{\ldots};
    \node (K666) at (0.95,-4.6) [circle, draw, fill=black!50, inner sep=0pt, minimum width=4pt]{};
    \node (KK666) at (0.75,-4.55) []{\tiny{$(2,q-2)$}};
    \node (K6666) at (0.6,-4.8) [circle, draw, fill=black!50, inner sep=0pt, minimum width=4pt]{};
    \node (KK6666) at (0.4,-4.75) []{\tiny{$(2,q-1)$}};
    %\node (L6) at (2,-4.5) []{\ldots};

    % Kanten

    \draw[-] (A666) to (A6666);
    %\draw[-] (J666) to (J6666);
    \draw[-] (K666) to (K6666);
    \draw[-] (A6) to (B6);
    \draw[-] (B6) to (C6);
    \draw[-] (C6) to (D6);
    \draw[-] (D6) to (E6);
    %    \draw[-] (E6) to (F6);
    \draw[-] (F6) to (G6);
    %    \draw[-] (G6) to (H6);
    \draw[-] (I6) to (B6);
    \draw[-] (I6) to (D6);
    %\draw[-] (I6) to (J6);
    \draw[-] (I6) to (K6);
    %\draw[-] (J6) to (J66);
    \draw[-] (K6) to (K66);
    %\draw[-] (C6) to (J6);
    \draw[-] (C6) to (K6);
    \draw[dashed] ([xshift=0.5]C6.north) to ([xshift=0.5]I6.south);
    \draw[dashed] ([xshift=-0.5]C6.north) to ([xshift=-0.5]I6.south);

  \end{tikzpicture}
  \caption{The diagram $T_{p,q,r}$} \label{def:TreeSing}
\end{figure}

The monodromy group of this singularity is given by the extended Weyl group $W$ attached to the diagram $\Gamma$. Denote by $\Lambda^*$ the set of vanishing cycles of $f$ and by $M$ the corresponding Milnor lattice. They are given by the generalized root system attached to $\Gamma$ and the corresponding root lattice.
Denote by $h_*:M \rightarrow M$ the (classical) monodromy operator of $f$. We can identify the latter one with a Coxeter transformation in $W$. As a consequence of Theorems \ref{thm:main} and \cite[Theorem 1.3]{BWY21} we obtain:

\begin{customthm}{\ref{thm:dist_bases}}
The set of all distinguished bases of vanishing cycles of $f$ is given by the set
$$
\{ (\delta_1 , \ldots, \delta_n) \in (\Lambda^*)^n \mid \spanz(\delta_1, \ldots , \delta_n)=M, ~s_{\delta_1} \cdots s_{\delta_n}=h_* \}.
$$
\end{customthm}
%\newpage
%\nocite{*}

\bibliography{mybibDerived2}

\providecommand{\bysame}{\leavevmode\hbox to3em{\hrulefill}\thinspace}
\providecommand{\MR}{\relax\ifhmode\unskip\space\fi MR }
\providecommand{\MRhref}[2]{%
  \href{http://www.ams.org/mathscinet-getitem?mr=#1}{#2}
}
\providecommand{\href}[2]{#2}
\begin{thebibliography}{10}

\bibitem{HTT07}
Lidia Angeleri~H\"{u}gel, Dieter Happel, and Henning Krause, \emph{Basic
  results of classical tilting theory}, Handbook of tilting theory, London
  Math. Soc. Lecture Note Ser., vol. 332, Cambridge Univ. Press, Cambridge,
  2007, pp.~9--13. \MR{2384605}

\bibitem{Arn73}
Vladimir~I. Arnold, \emph{Remarks on the method of stationary phase and on the
  {C}oxeter numbers}, Uspehi Mat. Nauk \textbf{28} (1973), no.~5(173), 17--44.
  \MR{0397777}

\bibitem{BDSW14}
Barbara Baumeister, Matthew Dyer, Christian Stump, and Patrick Wegener, \emph{A
  note on the transitive {H}urwitz action on decompositions of parabolic
  {C}oxeter elements}, Proc. Amer. Math. Soc. Ser. B \textbf{1} (2014),
  149--154. \MR{3294251}

\bibitem{BMW25}
Barbara Baumeister, Jon McCammond, and Patrick Wegener, \emph{The hyperbolic
  cover of an elliptic {W}eyl {G}roup}, arXiv e-prints (2025),
  arxiv:2411.06401.

\bibitem{BW18}
Barbara Baumeister and Patrick Wegener, \emph{A note on {W}eyl groups and root
  lattices}, Arch. Math. (Basel) \textbf{111} (2018), no.~5, 469--477.
  \MR{3859428}

\bibitem{BWY21}
Barbara Baumeister, Patrick Wegener, and Sophiane Yahiatene, \emph{Extended
  {W}eyl groups, {H}urwitz transitivity and weighted projective lines {I}:
  Generalities and the tubular case}, arXiv e-prints (2024), arxiv:1808.05083.

\bibitem{Bou02}
Nicolas Bourbaki, \emph{Lie groups and {L}ie algebras. {C}hapters 4--6},
  Elements of Mathematics (Berlin), Springer-Verlag, Berlin, 2002, Translated
  from the 1968 French original by Andrew Pressley. \MR{1890629}

\bibitem{Br07}
Kristian Br\"{u}ning, \emph{Thick subcategories of the derived category of a
  hereditary algebra}, Homology Homotopy Appl. \textbf{9} (2007), no.~2,
  165--176. \MR{2366948}

\bibitem{Deod82}
Vinay~V. Deodhar, \emph{On the root system of a {C}oxeter group}, Comm. Algebra
  \textbf{10} (1982), no.~6, 611--630. \MR{647210}

\bibitem{DL11}
Matthew~J. Dyer and Gus~I. Lehrer, \emph{Reflection subgroups of finite and
  affine {W}eyl groups}, Trans. Amer. Math. Soc. \textbf{363} (2011), no.~11,
  5971--6005. \MR{2817417}

\bibitem{Ebe07}
Wolfgang Ebeling, \emph{Functions of several complex variables and their
  singularities}, Graduate Studies in Mathematics, vol.~83, American
  Mathematical Society, Providence, RI, 2007, Translated from the 2001 German
  original by Philip G. Spain. \MR{2319634}

\bibitem{Ebe19}
\bysame, \emph{Distinguished bases and monodromy of complex hypersurface
  singularities}, Handbook of geometry and topology of singularities. {I},
  Springer, Cham, [2020] \copyright 2020, pp.~449--490. \MR{4261558}

\bibitem{EE98}
Henrik Eriksson and Kimmo Eriksson, \emph{Affine {W}eyl groups as infinite
  permutations}, Electron. J. Combin. \textbf{5} (1998), Research Paper 18, 32.
  \MR{1611984}

\bibitem{Gab74}
Andrei~M. Gabri\`elov, \emph{Dynkin diagrams of unimodal singularities},
  Funkcional. Anal. i Prilo\v{z}en. \textbf{8} (1974), no.~3, 1--6.
  \MR{0367274}

\bibitem{GL87}
Werner Geigle and Helmut Lenzing, \emph{A class of weighted projective curves
  arising in representation theory of finite-dimensional algebras},
  Singularities, representation of algebras, and vector bundles ({L}ambrecht,
  1985), Lecture Notes in Math., vol. 1273, Springer, Berlin, 1987,
  pp.~265--297. \MR{915180}

\bibitem{Hap01}
Dieter Happel, \emph{A characterization of hereditary categories with tilting
  object}, Invent. Math. \textbf{144} (2001), no.~2, 381--398. \MR{1827736}

\bibitem{HK16}
Andrew Hubery and Henning Krause, \emph{A categorification of non-crossing
  partitions}, J. Eur. Math. Soc. (JEMS) \textbf{18} (2016), no.~10,
  2273--2313. \MR{3551191}

\bibitem{Hum90}
James~E. Humphreys, \emph{Reflection groups and {C}oxeter groups}, Cambridge
  Studies in Advanced Mathematics, vol.~29, Cambridge University Press,
  Cambridge, 1990. \MR{1066460}

\bibitem{IS10}
Kiyoshi Igusa and Ralf Schiffler, \emph{Exceptional sequences and clusters}, J.
  Algebra \textbf{323} (2010), no.~8, 2183--2202. \MR{2596373}

\bibitem{IT09}
Colin Ingalls and Hugh Thomas, \emph{Noncrossing partitions and representations
  of quivers}, Compos. Math. \textbf{145} (2009), no.~6, 1533--1562.
  \MR{2575093}

\bibitem{Kac90}
Victor~G. Kac, \emph{Infinite-dimensional {L}ie algebras}, third ed., Cambridge
  University Press, Cambridge, 1990. \MR{1104219}

\bibitem{KL86}
Paul Kluitmann, \emph{Ausgezeichnete {B}asen erweiterter affiner
  {W}urzelgitter}, Bonner Mathematische Schriften, vol. 185, Universit\"{a}t
  Bonn, Mathematisches Institut, Bonn, 1987, Dissertation, Rheinische
  Friedrich-Wilhelms-Universit\"{a}t, Bonn, 1986. \MR{930668}

\bibitem{Kra12}
Henning Krause, \emph{Report on locally finite triangulated categories}, J.
  K-Theory \textbf{9} (2012), no.~3, 421--458. \MR{2955969}

\bibitem{Len99}
Helmut Lenzing, \emph{Coxeter transformations associated with
  finite-dimensional algebras}, Computational methods for representations of
  groups and algebras ({E}ssen, 1997), Progr. Math., vol. 173, Birkh\"{a}user,
  Basel, 1999, pp.~287--308. \MR{1714618}

\bibitem{LR16}
Joel~Brewster Lewis and Victor Reiner, \emph{Circuits and {H}urwitz action in
  finite root systems}, New York J. Math. \textbf{22} (2016), 1457--1486.
  \MR{3603073}

\bibitem{Looi80}
Eduard Looijenga, \emph{Invariant theory for generalized root systems}, Invent.
  Math. \textbf{61} (1980), no.~1, 1--32. \MR{587331}

\bibitem{McC15}
Jon McCammond, \emph{Dual euclidean {A}rtin groups and the failure of the
  lattice property}, J. Algebra \textbf{437} (2015), 308--343. \MR{3351966}

\bibitem{RI84}
Claus~Michael Ringel, \emph{Tame algebras and integral quadratic forms},
  Lecture Notes in Mathematics, vol. 1099, Springer-Verlag, Berlin, 1984.
  \MR{774589}

\bibitem{Ri94}
\bysame, \emph{The braid group action on the set of exceptional sequences of a
  hereditary {A}rtin algebra}, Abelian group theory and related topics
  ({O}berwolfach, 1993), Contemp. Math., vol. 171, Amer. Math. Soc.,
  Providence, RI, 1994, pp.~339--352. \MR{1293154}

\bibitem{Sai74}
Kyoji Saito, \emph{Einfach-elliptische {S}ingularit\"{a}ten}, Invent. Math.
  \textbf{23} (1974), 289--325. \MR{354669}

\bibitem{Sai85}
\bysame, \emph{Extended affine root systems. {I}. {C}oxeter transformations},
  Publ. Res. Inst. Math. Sci. \textbf{21} (1985), no.~1, 75--179. \MR{780892}

\bibitem{ST97}
Kyoji Saito and Tadayoshi Takebayashi, \emph{Extended affine root systems.
  {III}. {E}lliptic {W}eyl groups}, Publ. Res. Inst. Math. Sci. \textbf{33}
  (1997), no.~2, 301--329. \MR{1442503}

\bibitem{STW16}
Yuuki Shiraishi, Atsushi Takahashi, and Kentaro Wada, \emph{On {W}eyl groups
  and {A}rtin groups associated to orbifold projective lines}, J. Algebra
  \textbf{453} (2016), 249--290. \MR{3465355}

\bibitem{ST89}
Ernst Snapper and Robert~J. Troyer, \emph{Metric affine geometry}, second ed.,
  Dover Books on Advanced Mathematics, Dover Publications, Inc., New York,
  1989. \MR{1034484}

\bibitem{Lek}
Harm van~der Lek, \emph{Extended {A}rtin groups}, Singularities, {P}art 2
  ({A}rcata, {C}alif., 1981), Proc. Sympos. Pure Math., vol.~40, Amer. Math.
  Soc., Providence, RI, 1983, pp.~117--121. \MR{713240}

\bibitem{Lek2}
Harm van~der Lek, \emph{The homotopy type of complex hyperplane complements},
  Ph.D. thesis, University of Nijmegen, 1983.

\bibitem{PW17b}
Patrick Wegener, \emph{Hurwitz action in coxeter groups and elliptic weyl
  groups}, Ph.D. thesis, Universit\"at Bielefeld, 2017.

\bibitem{PW17a}
Patrick Wegener, \emph{On the {H}urwitz action in affine {C}oxeter groups}, J.
  Pure Appl. Algebra \textbf{224} (2020), no.~7, 106308. \MR{4058242}

\bibitem{WY19}
Patrick Wegener and Sophiane Yahiatene, \emph{A note on non-reduced reflection
  factorizations of coxeter elements}, Algebraic Combinatorics \textbf{3}
  (2020), no.~2, 465--469.

\end{thebibliography}
\bibliographystyle{amsplain}

\end{document}